\documentclass[a4paper]{article}

\usepackage[T1]{fontenc}
\usepackage[latin1]{inputenc}
\usepackage{amsfonts} 
\usepackage{amsbsy}
\usepackage{eucal}

\usepackage{amsthm}
\usepackage{amsxtra}
\usepackage[english]{babel}
\usepackage{latexsym,amssymb}
\usepackage{amsmath,amsfonts}
\usepackage{lmodern}
\usepackage{microtype}
\usepackage{multicol}

\theoremstyle{plain}
\newtheorem{theorem}{Theorem} 
\newtheorem{proposition}{Proposition}

\theoremstyle{plain}
\newtheorem{thm}{Theorem}
\newtheorem{lem}[thm]{Lemma}
\newtheorem{coro}[thm]{Corollary}

\newtheorem{prop}[thm]{Proposition}

\theoremstyle{definition}

\newtheorem{defi}[thm]{Definition}

\newcommand{\R}{\mathbb{R}}
\newcommand{\Z}{\mathbb{Z}}
\newcommand{\N}{\mathbb{N}}
\renewcommand{\H}{\mathbb{H}}
\newcommand{\C}{\mathbb{C}}

\newcommand{\abs}[1]{\left|#1\right|}

\newcommand*{\rom}[1]{\romannumeral}
\renewcommand{\Re}{{Re}}

\DeclareMathOperator{\fl}{flat}

\DeclareMathOperator{\GL}{GL}
\DeclareMathOperator{\SO}{SO}

\DeclareMathOperator{\Spin}{Spin}

\DeclareMathOperator{\Tr}{Tr}
\DeclareMathOperator{\tr}{tr}

\DeclareMathOperator{\Vol}{Vol}
\DeclareMathOperator{\rank}{rank}
\DeclareMathOperator{\Id}{Id}
\DeclareMathOperator{\Ad}{Ad}
\DeclareMathOperator{\ad}{ad}
\DeclareMathOperator{\T}{T}
\DeclareMathOperator{\Norm}{Norm}
\DeclareMathOperator{\cent}{Centr}
\DeclareMathOperator{\spec}{spec}

\DeclareMathOperator{\prim}{prime}

\DeclareMathOperator{\Ind}{Ind}

\DeclareMathOperator{\Cl}{Cl}
\DeclareMathOperator{\AS}{AS}
\DeclareMathOperator{\gr}{gr}
\numberwithin{equation}{section}
\numberwithin{thm}{section}
\providecommand{\keywords}[1]{{\noindent\small\textbf{Keywords} --- #1.}}

\renewcommand{\thethm}{%
\ifnum\value{subsection}=0
\thesection
\else
\thesubsection
	\fi
	.\arabic{thm}%
}

\providecommand{\contact}{{
	\bigskip
	\small

	\noindent
	\textsc{P.~Spilioti}, 
	Department of Mathematics, Ny Munkegade 118,  8000, Aarhus C,
     Denmark
	\par\noindent\nopagebreak
	\textit{E-mail address:} \href{mailto:spilioti@math.au.dk}
	{\texttt{spilioti@math.au.dk}}
}}

\begin{document}

\title{Functional equations of Selberg and Ruelle zeta functions for non-unitary twists}
\author{Polyxeni Spilioti}

\maketitle
\begin{abstract}
We consider the dynamical zeta functions of Selberg and Ruelle
associated with the geodesic flow on a compact odd-dimensional hyperbolic manifold.
These dynamical zeta functions are defined for a complex variable $s$ in some right-half plane of $\C$.
In \cite{Spilioti2018}, it was proved  that they admit a meromorphic
continuation to the whole complex plane. In this paper, we establish functional equations for them, relating their values at $s$ with those
at $-s$. We prove also a determinant representation of the zeta functions, using the regularized determinants of certain twisted differential operators.
\end{abstract}


\section{Introduction}
\label{intro}

We consider the twisted Selberg and Ruelle zeta functions of compact hyperbolic
manifolds $X$ of odd dimension $d$.
Let $G,K$ be either $G=\SO^{0}(d,1)$, $K=\SO(d)$ or $G=\Spin(d,1)$, $K=\Spin(d)$. Then, $K$ is a maximal compact subgroup of $G$.
Let $\widetilde{X}:=G/K$. $\widetilde{X}$ can be equipped with a $G$-invariant metric, which is unique up to scaling 
and is of constant negative curvature.
If we normalize this metric such that it has constant negative curvature $-1$, then
$\widetilde{X}$ equipped with this metric is isometric to the $d$-dimensional hyperbolic space $\H^{d}$.
Let $\Gamma$ be a discrete, torsion-free subgroup of $G$ such that $\Gamma\backslash G$ is compact. 
This means that $\Gamma$ has no elements of finite order.
$\Gamma$ acts by isometries on $\widetilde X$ and $X=\Gamma\backslash \widetilde X$ is a
compact oriented hyperbolic manifold of dimension $d$. 
This is a case of a locally symmetric space of non-compact type of real rank 1. This means that 
in the Iwasawa decomposition $G=KAN$ of $G$, $A$ is a multiplicative torus of dimension 1, i.e., $A\cong\R^+$. 

For a given $\gamma\in\Gamma$, we denote by $[\gamma]$ the $\Gamma$-conjugacy class
of $\gamma$. If $\gamma\neq e$, then there is a unique closed geodesic 
$c_{\gamma}$ associated with $[\gamma]$. We denote by $l(\gamma)$ the length of
$c_{\gamma}$.
The conjugacy class $[\gamma]$ is called prime if 
there exist no $k\in\N$ with $k>1$ and $\gamma_{0}\in\Gamma$ such that $\gamma=\gamma_{0}^{k}$.
The prime closed geodesics correspond to the prime conjugacy classes.
These geodesics trace out their image exactly once. 
Let $M:=\cent_{K}(A)$ be the centralizer of $A$ in $K$. 
Since $\Gamma$ is a cocompact subgroup of $G$, every element $\gamma\in \Gamma-\{e\}$ is  hyperbolic.
By \cite[Lemma 6.5]{Wa}, there exist a $g\in G$, a $m_{\gamma} \in M$, and an 
$a_{\gamma} \in A$, such that $ g^{-1}\gamma g=m_{\gamma}a_{\gamma}$.
The element $m_{\gamma}$ is determined up to conjugacy in $M$, and the element $a_{\gamma}$ 
is uniquely determined by $\gamma$.

Let $\mathfrak{g},\mathfrak{a}$ be the Lie algebras of $G$ and $A$, respectively.
Let $\Delta^{+}(\mathfrak{g},\mathfrak{a})$ be the set of positive roots of $(\mathfrak{g},\mathfrak{a})$.
Then, $\Delta^{+}(\mathfrak{g},\mathfrak{a})$ consists of a single root 
$\alpha$. Let $\mathfrak{g}_{\alpha}$ be the corresponding root space.
Let $\overline{\mathfrak{n}}$ be the negative root space of $(\mathfrak{g},\mathfrak{a})$.
Let $S^k(\Ad(m_\gamma a_\gamma)|_{\overline{\mathfrak{n}}})$ be the $k$-th
symmetric power of the adjoint map $\Ad(m_\gamma a_\gamma)$ restricted to $\overline{\mathfrak{n}}$
and $\rho$ be defined as $\rho:=\frac{1}{2}\dim(\mathfrak{g}_\alpha)\alpha$. 
We define the twisted zeta functions associated
with unitary irreducible representations $\sigma$ of $M$ and 
finite-dimensional representations $\chi$ of $\Gamma$.
The twisted Selberg zeta function $Z(s;\sigma,\chi)$ is defined for $s\in\C$ by the infinite product
\begin{equation*}
Z(s;\sigma,\chi):=\prod_{\substack{[\gamma]\neq e,\\ [\gamma]\prim}} \prod_{k=0}^{\infty}\det\Big(\Id-\big(\chi(\gamma)
\otimes\sigma(m_\gamma)\otimes S^k(\Ad(m_\gamma a_\gamma)|_{\overline{\mathfrak{n}}})\big)e^{-(s+|\rho|)l(\gamma)}\Big).
\end{equation*}
$Z(s;\sigma,\chi)$ converges absolutely and uniformly on compact subsets of some half-plane of $\C$ (\cite[Proposition 3.4]{Spilioti2018}).
The twisted Ruelle zeta function $R(s;\sigma,\chi)$ is defined for $s\in\C$ by the infinite product
\begin{equation*}
 R(s;\sigma,\chi):=\prod_{\substack{[\gamma]\neq{e}\\ [\gamma]\prim}}\det(\Id-(\chi(\gamma)\otimes\sigma(m_{\gamma}))e^{-sl(\gamma)}).
\end{equation*}
$R(s;\sigma,\chi)$ converges absolutely and uniformly on compact subsets of some half-plane of $\C$ 
(\cite[Proposition 3.5]{Spilioti2018}).

The connection of the Ruelle zeta function to spectral invariants such as the analytic 
torsion has been studied for example by 
Fried (\cite{Fried}), Bunke and Olbrich (\cite{BO}), Wotzke (\cite{Wo}), M\"{u}ller (\cite{M2})
for compact hyperbolic manifolds, under certain assumptions on the representation of the fundamental group of the manifold.
For the case of a hyperbolic manifold of finite volume, we refer the reader to the work of 
Park (\cite{Park}) 
and Pfaff (\cite{pfaff2}, \cite{pfaff1}).
A far more advanced study of the dynamical zeta functions of locally symmetric manifolds of higher rank is 
due to Moscovici and Stanton in \cite{MS91}, Deitmar in \cite{Deit1}, Schen in \cite{Shen2018} and
Moscovici, Stanton and Frahm in \cite{MSF}.
Fedosova and Pohl in \cite{Fedosova2020}, studied the Selberg zeta function on hyperbolic surfaces
for geometrically finite, non-elementary Fuchsian groups $\Gamma$ and finite-dimensional representations
with non-expanding cusp monodromy. They use transfer operators techniques
to prove the meromorphic continuation of the zeta function.

For hyperbolic manifolds of odd dimension, Fried
in \cite{Fried} considered 
an orthogonal representation $\varrho\colon\Gamma\rightarrow \rm O(m)$ of $\Gamma$.
Using the Selberg trace formula for the heat operator $e^{-t\Delta_{j}}$, 
where $\Delta_{j}$ is the Hodge Laplacian  on $j$-forms on $X$
with values in the flat vector bundle  $E_{\varrho}$ associated with $\varrho$,
he proved the meromorphic continuation of the zeta functions to the whole complex plane $\C$, as well as functional equations 
for the Selberg zeta function (\cite[p. 531-532]{Fried}). Under the assumption that
$\varrho$ is acyclic, i.e.,  the cohomology with coefficients in the local system 
defined by $\rho$ vanish for all $j$,
the Ruelle zeta function, defined for Re$(s)>d-1$ by 
\begin{equation*}
 R(s;\varrho):=\prod_{\substack{[\gamma]\neq e,\\ [\gamma]\prim}}{\det(\Id-\varrho(\gamma)e^{-sl(\gamma)}}),
\end{equation*}
admits a meromorphic extension to $\C$.
In addition, it is regular at $s=0$ and
\begin{equation*}
\abs{{R(0;\varrho)}}=T_X(\varrho)^2,
\end{equation*}
where $T_X(\varrho)$ is the Ray-Singer analytic torsion defined as in \cite{RS}.

For unitary representations $\chi$, the dynamical zeta functions have been studied by Bunke and Olbrich in 
\cite{BO} for all locally symmetric spaces of real rank 1. They proved that the zeta functions admit a meromorphic
continuation to the whole complex plane and they satisfy functional equations.

Wotzke in \cite{Wo} extended this result for representations of $\Gamma$, which 
are not necessary unitary, but very special ones.
In particular, he considered a compact 
odd-dimensional hyperbolic manifold and a finite-dimensional irreducible representation 
$\tau\colon G\rightarrow \GL(V)$ of $G$, such that $\tau\neq\tau_{\theta}$, where 
$\tau_{\theta}=\tau\circ\theta$ and
$\theta$ denotes the Cartan involution of $G$.
Under these assumptions, he proved that the Ruelle zeta function 
admits a meromorphic continuation to the whole 
complex plane. In addition, it is regular at $s=0$ and
\begin{equation*}
 \abs{{R(0;\tau)}}=T_X(\tau)^2.
\end{equation*}
Wotzke's method is based on the fact that if one considers the restrictions $\tau|_{K}$ and $\tau|_{\Gamma}$ of $\tau$
to $K$ and $\Gamma$, respectively, there is an isomorphism 
between the locally homogeneous vector bundle $E_{\tau}$ over $X$ 
associated with $\tau|_{K}$  and the flat vector bundle $E_{\fl}$ over $X$
associated with $\tau|_{\Gamma}$.
By \cite[Lemma 3.1]{MM}, a Hermitian 
fiber metric in $E_{\tau}$ descends to a fiber metric in $E_{\fl}$. Therefore, all tools from harmonic analysis
on locally symmetric spaces
are available.

In our case the representation of $\Gamma$ is arbitrary.
Hence, these tools are no longer applicable.
In \cite{Spilioti2018}, it is proved
that the twisted Selberg and Ruelle zeta functions associated with an arbitrary finite-dimensional
representation of $\Gamma$
admit a meromorphic continuation to the whole complex plane.
In the present  paper, we extend the results of \cite{BO} to the case of the non-unitary twist $\chi$.
We prove functional equations for the twisted dynamical zeta functions, 
relating their values at $s$ with those
at $-s$. Moreover, we prove a determinant formula, which 
expresses the zeta function in terms of regularized determinants
of certain twisted Laplace-type operators.
We state here our main results. In the functional equation (1.1) below,
$P_{\sigma}(s)$ denotes the Plancherel polynomial  associated 
with $\sigma\in\widehat{M}$
(see Section 2).
\begin{theorem}
 The Selberg zeta function $Z(s;\sigma,\chi)$ satisfies the following functional equation
\begin{equation}
 \frac{Z(s;\sigma,\chi)}{Z(-s;\sigma,\chi)}=\exp\bigg(-4\pi\dim(V_{\chi})\Vol(X)\int_{0}^{s}P_{\sigma}(r)dr\bigg).
\end{equation}
\end{theorem}
Let $M':=\Norm_{K}(A)$ be the normalizer of $A$ in $K$.
We define the restricted Weyl group by the quotient $ W_{A}:=M'/M$. 
Then, $W_{A}$ has order $2$.
$W_{A}$ acts on $\widehat{M}$ by 
$(w\sigma)(m):=\sigma(m_{w}^{-1}mm_{w})$,
where $w$ is a non-trivial element in $W_{A}$, 
$m_{w}$ is a representative of $w$ in $M'$ and 
$m\in M$.
In this paper, we distinguish always the case of $\sigma\in\widehat{M}$ being non-Weyl invariant,
i.e., $w\sigma\neq\sigma$.
In that case, we define the \textit{super}
Ruelle zeta function by
\begin{equation*}
 R^{s}(s;\sigma,\chi):=\frac{R(s;\sigma,\chi)}{R(s;w\sigma,\chi)}.
\end{equation*}
The \textit{twisted} Dirac operator associated with an arbitrary representation of $\Gamma$,
acting on smooth sections of twisted vector bundles is defined in 
\cite[Section 6]{Spilioti2018} (see also Section 4 in the present paper). As in the case of the twisted Bochner-Laplace operator (\cite{Spilioti2018}), this operator 
acts locally as the identity operator on smooth sections of the flat vector bundle.
In Section 4, we define the eta function of the twisted Dirac operator and
further its eta invariant. In addition, we  prove a formula (Lemma 4.6), which generalizes the
usual integral representation of the eta function to the case of non-unitary twists.
In the right-hand side of the functional equation (1.2) below, $\widetilde{D^{\sharp}_{\chi}(\sigma)}$
is a twisted Dirac operator (see p. 30) and 
the term $\eta(\widetilde{D^{\sharp}_{\chi}(\sigma)})$
is defined in terms of the eta function (Definition 4.5, Section 4) at zero of twisted Dirac operators
(for more details see p. 30-31).
\begin{theorem}
The super Ruelle zeta function associated with a non-Weyl invariant representation $\sigma\in\widehat{M}$ 
satisfies the functional equation
\begin{equation}
 R^{s}(s;\sigma,\chi) R^{s}(-s;\sigma,\chi)=e^{2i\pi\eta(\widetilde{D^{\sharp}_{\chi}(\sigma)})}.
\end{equation}
\end{theorem}
The determinant formula for the twisted Ruelle zeta function (Proposition 7.9) is a direct consequence of the determinant formula for the twisted Selberg zeta function (Theorem 7.8). 
We denote by $\sigma_{p}$ the standard representation of $M$ in $\Lambda^{p}\R^{d-1}\otimes \C$.
Let $|\rho|$ be the norm of $\rho$, induced by
the normalized Killing form (see (2.1), Section 2).
We consider the operators $A^{\sharp}_{\chi}(\sigma_{p}\otimes\sigma)+(s+\vert\rho\vert-p)^{2})$, 
acting on smooth sections of graded twisted vector bundles
(see \cite[p. 174--175]{Spilioti2018}, Section 4 and Section 7 in the present paper).
These operators are elliptic, differential operators of second order, 
which are not self adjoint. Since the vector bundles we consider are $\Z_{2}$-graded,
we consider  graded regularized determinants of these operators
(see Section 7).
\begin{proposition}
The Ruelle zeta function has the representation 
\begin{align*}
R(s;\sigma,\chi)=\notag&\prod_{p=0}^{d-1}{{\det}_{\gr}(A^{\sharp}_{\chi}(\sigma_{p}\otimes\sigma)+(s+\vert\rho\vert-p)^{2})}^{(-1)^{p}}\\
&\exp\bigg((-1)^{\frac{d-1}{2}+1}\pi(d+1)\dim(V_{\sigma})\dim (V_{\chi})\frac{\Vol(X)}{\Vol(S^{d})}s\bigg).
\end{align*}
\end{proposition}
This paper is organized as follows. In Section 2, we introduce the basic setup.
In Section 3, we define the twisted dynamical zeta functions. Section 4 is devoted to the study of the twisted Dirac operator and the definition of the associated eta function. The functional equations for the Selberg zeta function are derived in Section 5. In Section 6, we prove the functional equations for the Ruelle and super Ruelle zeta function. Finally, in Section 7 we prove the determinant formula for the Ruelle zeta function.

\section{Preliminaries}

\subsection{Representation theory of Lie groups}
We introduce here our algebraic setting and fix some notation. For further details we refer the reader to
\cite[Section 2]{Spilioti2018}.
Let $n\in\N_{>0}$.
Let $d=2n+1$  be an odd integer.
We consider the universal coverings $G=\Spin(d,1)$ of $\mathrm{SO}^0(d,1)$ and $K=\Spin(d)$ of $\mathrm{SO}(d)$, respectively.
We set  $\widetilde{X}:=G/K$. 
Let $\mathfrak{g},\mathfrak{k}$ be the Lie algebras of $G$ and $K$, respectively. We consider the following complexification
$\mathfrak{g}_{\C}:=\mathfrak{g}\oplus i\mathfrak{g}$.
Let 
$
 \mathfrak{g}=\mathfrak{k}\oplus\mathfrak{p}
$
be the Cartan decomposition of $\mathfrak{g}$ with respect to the Cartan involution of $G$. There exists a canonical isomorphism $T_{eK}\widetilde{X}\cong \mathfrak{p}$.
Let $\mathfrak{a}$ be a maximal abelian subalgebra of $\mathfrak{p}$.
Let $A$ be the subgroup of $G$ with Lie algebra $\mathfrak{a}$. Let $M:=\cent_{K}(A)$ be the centralizer of $A$ in $K$. 
Then, $M=\Spin(d-1)$. 
Let $\mathfrak{m}$ be the Lie algebra of $M$ and $\mathfrak{m}_{\C}:=\mathfrak{m}\oplus i\mathfrak{m}$ the  complexification of $\mathfrak{m}$. 
Let $\mathfrak{b}$, $\mathfrak{h}$ be a Cartan subalgebra of $\mathfrak{m}$ and $\mathfrak{g}$, respectively.
Let $B(X,Y)$ be the Killing form on $\mathfrak{g}\times \mathfrak{g}$ defined by $ B(X,Y)=\Tr (\ad(X)\circ \ad(Y))$.
It is a symmetric bilinear form. 
We consider the symmetric bilinear form
\begin{equation}
 \langle Y_1, Y_2\rangle:=\frac{1}{2(d-1)}B(Y_1,Y_2),\quad Y_1,Y_2 \in \mathfrak{g}.
\end{equation}
The restriction of $\langle\cdot,\cdot\rangle$ to $\mathfrak{p}$ defines
an inner product on $\mathfrak{p}$ and
therefore induces a $G$-invariant Riemannian metric on $\widetilde{X}$, which 
has constant curvature $-1$. $\widetilde{X}$ with this metric
is isometric to the $d$-dimensional real hyperbolic space $\H^{d}$.
Let $\Gamma_{1}$ be a torsion free cocompact discrete subgroup of $\mathrm{SO}^0(d,1)$.
We assume that $\Gamma_{1}$ can be lifted to a subgroup $\Gamma$ of $G$. Then, $X:=\Gamma\backslash \widetilde{X}$ is a
compact hyperbolic manifold of odd dimension $d$.

Let $G=KAN$ be the standard Iwasawa decomposition of $G$. 
Let $\Delta^{+}(\mathfrak{g},\mathfrak{a})$ be the set of positive roots of $(\mathfrak{g},\mathfrak{a})$.
Then, $\Delta^{+}(\mathfrak{g},\mathfrak{a})$ consists of a single root 
$\alpha$.
Let $\mathfrak{g}_{\alpha}$ be the corresponding root space.
We define 
\begin{equation}
\rho:=\frac{1}{2}\dim(\mathfrak{g}_\alpha)\alpha.
\end{equation}
Let $\Delta^+(\mathfrak{m}_{\C},\mathfrak{b})$ be the set of the positive roots of the system $(\mathfrak{m}_{\C},\mathfrak{b})$.
We define
\begin{equation}
\rho_{\mathfrak{m}}:=\frac{1}{2}\sum_{\alpha\in \Delta^+(\mathfrak{m}_{\C},\mathfrak{b})}\alpha. 
\end{equation}
Let $\Delta^{+}(\mathfrak{g}_{\C},\mathfrak{h})$ be the set of the positive roots of the system $(\mathfrak{g}_{\C},\mathfrak{h})$.
We define
\begin{equation}
 \rho_{\mathfrak{g}}:=\frac{1}{2}\sum_{\alpha\in\Delta^{+}(\mathfrak{g}_{\C},\mathfrak{h})}\alpha.
\end{equation}
Let $\nu_{\sigma}$ be the highest weight of $\sigma$. 
We have
\begin{equation}
   \nu_{\sigma}=(\nu_{1},\ldots,\nu_{n-1},\nu_{n}),
\end{equation}
where $\nu_{1}\geq\ldots\geq\nu_{n-1}\geq\lvert\nu_{n}\rvert$ and $\nu_{i},i=1,\ldots,n$ are all half integers (\cite[p. 20]{BO}).

Let $W_{A}$ be the restricted Weyl group defined by 
$
 W_{A}:=M'/M,
$
where  $M'=\Norm_{K}(A)$ is the normalizer of $A$ in $K$.
$W_{A}$ has order 2. Let $\widehat{M}$ be the set of the equivalent classes of irreducible unitary representations of $M$.
For a non-trivial element $w$ of $W_{A}$, 
the action of $W_{A}$ on $\widehat{M}$ is defined by
\begin{equation*}
(w\sigma)(m):=\sigma(m_{w}^{-1}mm_{w}),\quad m\in M, \sigma\in\widehat{M},
\end{equation*}
where $m_{w}$ is a representative of $w$ in $M'$.
We  distinguish the following two cases for this action:
\begin{itemize}
 \item  {\bf case (a)}: \textit{$\sigma$ is invariant under the action of the restricted Weyl group $W_{A}$;}
 \item {\bf case (b)}: \textit{$\sigma$ is not invariant under the action of the restricted Weyl group $W_{A}$.}
\end{itemize}
 
\subsection{The principal series representation}

Let $P=MAN$ be the standard parabolic subgroup of $G$.
Let $\rho$ be defined by (2.2). Let  $|\rho|$ be its norm, induced by
the symmetric bilinear form, defined by (2.1).
For $(\sigma,V_{\sigma})\in \widehat{M}$,
we define the space $\mathcal{H}_{\sigma,\lambda}$ of continuous functions on $G$ by
\begin{equation*}
 \mathcal{H}_{\sigma,\lambda}:=\{f\in C(G,V_{\sigma})\colon f(gman)=e^{-(i\lambda+|\rho|)\log a}\sigma^{-1}(m)f(g), \forall g\in G, \forall man\in P\},
\end{equation*}
where $\lambda\in\C$, with norm 
\begin{equation}
 \lVert f \rVert_{c}=\int_{K}\lVert f(k) \rVert^{2}dk.
\end{equation}
The principal series representation are defined by 
\begin{equation}
 \pi_{\sigma,\lambda}:=\Ind_{P}^{G}(\sigma\otimes e^{i\lambda}\otimes \Id),
\end{equation}
with representation space the Hilbert space, obtained by completion of $\mathcal{H}_{\sigma,\lambda}$ with respect to the norm 
$\lVert \cdot\rVert_{c}$ in (2.6).
For $f\in\mathcal{H}_{\sigma,\lambda}$, the action of $G$ on $f$ is given by $ \pi_{\sigma,\lambda}(g)f(g')=f(g^{-1}g')$.
Let $\mathfrak{a}_{\C}^{*}$ be the space of the linear functionals on $\mathfrak{a}_{\C}$.
In the definition of the space $ \mathcal{H}_{\sigma,\lambda}$,
$\lambda$ is a complex number. $\mathfrak{a}_{\C}^{*}$ 
is identified with $\C$, using the positive root.
If $\lambda\in\R$, then the representation $\pi_{\sigma,\lambda}$ is unitary.

\subsection{Plancherel measure}
Let $\mu_{PL}(\pi_{\sigma,\lambda})$ be the Plancherel measure, viewed as a measure on the set of 
the principal series representations $\pi_{\sigma,\lambda}$.
Since $\rank(G)>\rank(K)$, by classical result of Harish-Chandra (\cite{HC2}), the set of the 
discrete series representations of $G$ is empty.
By \cite[Theorem 13.2]{Knapp}, there exists a constant 
$c(n)\in\R$, $c(n)\neq 0$, such that
\begin{equation*}
 d\mu_{PL}(\pi_{\sigma,\lambda})=-c(n)Q_{\sigma}(i\lambda)d\lambda,
\end{equation*}
where $Q_{\sigma}(i\lambda)$ is the Plancherel polynomial
given by
\begin{equation}
 Q_{\sigma}(i\lambda)=\prod_{\alpha\in\Delta^{+}(\mathfrak{g}_{\C},\mathfrak{h})}\frac{\langle i\lambda+\nu_{\sigma}+\rho_{\mathfrak{m}},\alpha\rangle}{\langle \rho_{\mathfrak{g}},\alpha\rangle},
\end{equation}
where $\langle\cdot,\cdot\rangle$ is defined by (2.1).
Let $z=i\lambda\in\C$.
By \cite[eq. (2.33), (2.34)]{Pf},
\begin{equation}
 c(n)=\frac{(-1)^{n+1}}{2\Vol(S^{d})},
\end{equation}
where $\Vol(S^{d})$ denotes the volume of the 
$d$-dimensional Euclidean unit sphere.
By \cite[p. 264-265]{Mi2}, $Q_{\sigma}(z)$ is an even polynomial of $z$
and hence $Q_{\sigma}(z)=Q_{\sigma}(-z)$.
We set
\begin{equation*}
P_{\sigma}(z)=-c(n)Q_{\sigma}(z).
\end{equation*}
\textbf{Remark:}
We note that in the definition of the Plancherel measure in \cite[2.6]{Spilioti2018},
the constant $c(n)$ is missing.

\section{The dynamical zeta functions of Ruelle and Selberg}

In this section we define the twisted Selberg and Ruelle zeta functions associated with the geodesic flow on the
unit sphere  bundle $S(X)$ of $X$. Here, we identify $S(X)=\Gamma\backslash G/M$. For more details about this identification, we refer the reader to \cite[Subsection 1.1.1 and Section 3.1]{BO}.
There exists a 1-1 correspondence between the closed geodesics on a manifold $X$ with negative sectional curvature 
and the non-trivial conjugacy classes of the fundamental group $\pi_{1}(X)$ of $X$ (\cite{GKM}).
The hyperbolic elements of $\Gamma$ are the semisimple elements of this group.
Since $\Gamma$ is a cocompact and torsion free, every element $\gamma\in\Gamma-\{e\}$ is hyperbolic.
Let $[\gamma]$ be the conjugacy class of $\gamma$, defined by 
$[\gamma]:=\{ \gamma'\gamma(\gamma')^{-1}:\gamma'\in \Gamma\}$.
We denote by $c_{\gamma}$ the closed geodesic on $X$ associated with the hyperbolic conjugacy class $[\gamma]$,
and by $l(\gamma)$ the length of $c_{\gamma}$.
An element $\gamma\in\Gamma$ is called primitive if there exists no $n\in\N$ with $n>1$ and $\gamma_{0}\in\Gamma$ such that $\gamma=\gamma_{0}^{n}$.
A primitive element $\gamma_{0}\in\Gamma$ corresponds to a prime closed geodesic on $X$.
This is a geodesic of minimal length.
Hence, if a hyperbolic element $\gamma$ in $\Gamma$ is generated by a primitive element $\gamma_{0}$, then
there exists a $n_{\Gamma}(\gamma)\in \N$ such that $\gamma=\gamma_{0}^{n_{\Gamma}(\gamma)}$ and the corresponding 
closed geodesic is of length $l(\gamma)=n_{\Gamma}(\gamma)l(\gamma_{0})$.

We define the dynamical zeta functions for an
arbitrary finite-dimensional representation $\chi$ of $\Gamma$.
\begin{defi}
Let $\chi\colon\Gamma\rightarrow \GL(V_{\chi})$ be a finite-dimensional representation of $\Gamma$ and $\sigma\in \widehat{M}$.
The twisted Selberg zeta function $Z(s;\sigma,\chi)$ for 
$X$ is defined  by the infinite product
\begin{equation}
Z(s;\sigma,\chi):=\prod_{\substack{[\gamma]\neq{e}\\ [\gamma]\prim}} \prod_{k=0}^{\infty}\det\big(\Id-(\chi(\gamma)\otimes\sigma(m_\gamma)\otimes S^k(\Ad(m_\gamma a_\gamma)|_{\overline{\mathfrak{n}}})) e^{-(s+|\rho|)\lvert l(\gamma)}\big),
\end{equation}
where $s\in \C$, $\overline{\mathfrak{n}}$ is the negative root space of $(\mathfrak{g},\mathfrak{a})$ and
$S^k(\Ad(m_\gamma a_\gamma)_{\overline{\mathfrak{n}}})$ denotes the $k$-th
symmetric power of the adjoint map $\Ad(m_\gamma a_\gamma)$ restricted to $\mathfrak{\overline{n}}$.
\end{defi}
By \cite[Proposition 3.4]{Spilioti2018}, there exists a positive constant $c$, such that the infinite product in (3.1) converges 
absolutely and uniformly on compact subsets of the half-plane Re$(s)>c$.
\begin{defi} Let $\chi\colon\Gamma\rightarrow \GL(V_{\chi})$ be a finite-dimensional representation of $\Gamma$ and $\sigma\in \widehat{M}$.
The twisted Ruelle zeta function $ R(s;\sigma,\chi)$ for $X$ is defined by the infinite product
\begin{equation}
 R(s;\sigma,\chi):=\prod_{\substack{[\gamma]\neq{e}\\ [\gamma]\prim}}\det\big(\Id-(\chi(\gamma)
 \otimes\sigma(m_{\gamma}))e^{-sl(\gamma)}\big),
\end{equation}where $s\in\C$.
\end{defi}
By \cite[Proposition 3.5]{Spilioti2018}, there exists a positive constant $r$, such that the infinite product in (3.2) converges 
absolutely and uniformly on compact subsets of the half-plane Re$(s)>r$.
We consider now the case in which $\sigma$ is not invariant under the action of the restricted Weyl group $W_{A}$. This is case (b) (Section 2).
We recall from \cite[Section 3]{Spilioti2018} that in case (b), for a non-trivial element $w$ of $W_{A}$,
we define in addition the symmetrized zeta function
 \begin{equation}
 S(s;\sigma,\chi):=Z(s;\sigma,\chi)Z(s;w\sigma,\chi);
\end{equation}
 the super zeta function
 \begin{equation}
 Z^{s}(s;\sigma,\chi):=\frac{Z(s;\sigma,\chi)}{Z(s;w\sigma,\chi)};
\end{equation}
and the super Ruelle zeta function
\begin{equation*}
 R^{s}(s;\sigma,\chi):=\frac{R(s;\sigma,\chi)}{R(s;w\sigma,\chi)}.
\end{equation*}
We put
\begin{equation}
 L(\gamma;\sigma,\chi):=\frac{\tr(\chi(\gamma)\otimes\sigma(m_{\gamma}))e^{-|\rho|l(\gamma)}}
 {\det{(\Id-\Ad(m_{\gamma}a_{\gamma})_{\overline{n}})}}.
\end{equation}

By \cite[Lemma 3.9]{Spilioti2018}, the logarithmic derivatives $L(s), L_{S}(s), L^{s}(s)$ of the Selberg, symmetrized and super zeta function, respectively,
are given by
\begin{equation*}\label{log der selberg}
 L(s):=\frac{d}{ds}\log(Z(s;\sigma,\chi))=\sum_{[\gamma]\neq{e}}\frac{l(\gamma)}{n_{\Gamma}(\gamma)}L(\gamma;\sigma,\chi)
e^{-sl(\gamma)};
\end{equation*}
 \begin{equation}\label{f:log der symmetrized}
 L_{S}(s):=\frac{d}{ds}\log(S(s;\sigma,\chi))=\sum_{[\gamma]\neq{e}}\frac{l(\gamma)}{n_{\Gamma}(\gamma)}L(\gamma;\sigma+w\sigma,\chi)
e^{-sl(\gamma)};
\end{equation} 
 \begin{equation}\label{f:log der super}
 L^{s}(s):=\frac{d}{ds}\log(Z^{s}(s;\sigma,\chi))=\sum_{[\gamma]\neq{e}}\frac{l(\gamma)}{n_{\Gamma}(\gamma)}L(\gamma;\sigma-w\sigma,\chi)
e^{-sl(\gamma)}.
\end{equation}

\section{The eta function of the twisted Dirac operator}

Let $\chi\colon\Gamma\rightarrow \GL(V_{\chi})$ be a finite-dimensional representation of $\Gamma$ and $\sigma\in \widehat{M}$.
In this section, we define the eta function associated with the twisted Dirac operator, which will be denoted by
$D^{\sharp}_{\chi}(\sigma)$. This operator is associated with the representations $\sigma$ and  $\chi$.
It is an elliptic first order differential operator, but no longer self-adjoint.
We recall from \cite[Section 6]{Spilioti2018} the definition of $D^{\sharp}_{\chi}(\sigma)$.

Let $s$ be the spin representation of $\Spin(d)$ with representation space $S$.
Let $R(K)$, $R(M)$ be the representation rings over $\Z$ of $K$ and $M$, respectively. 
Let $i^{*}\colon R(K)\rightarrow R(M)$ be the pullback of the embedding
$i\colon M\hookrightarrow K$.
By \cite[Proposition 1.1, (3)]{BO}, there exists 
 a unique element  $\tau(\sigma)\in \widehat{K}$ and a splitting 
$
 s \otimes \tau(\sigma)=\tau^{+}(\sigma)\oplus \tau^{-}(\sigma)$,
where $\tau^{+}(\sigma),\tau^{-}(\sigma)\in R(K)$,
such that $
   \sigma+w\sigma=i^{*}(\tau^{+}(\sigma)-\tau^{-}(\sigma)).
$
Let $V_{\tau(\sigma)}$ be the representation space of $\tau(\sigma)$.
We define the representation 
\begin{equation*}
 \tau_{s}(\sigma):= s\otimes\tau(\sigma)
\end{equation*}
of $K$ with representation space $V_{\tau_{s}(\sigma)}=S\otimes V_{\tau(\sigma)}$.
We consider the homogeneous vector bundle $\widetilde{E}_{\tau(\sigma)}$ defined by
\begin{equation*}
\widetilde{E}_{\tau(\sigma)}:=G\times_{\tau(\sigma)}V_{\tau(\sigma)}\rightarrow\widetilde{X}
\end{equation*}
Let $\nabla^{\tau(\sigma)}$ be the canonical connection in  $\widetilde{E}_{\tau(\sigma)}$.
Let $E_{s}$ be the spinor bundle over $\widetilde{X}$ associated with $s$ and equipped with the spin connection $\nabla$.
The vector bundle $\widetilde{E}_{\tau_{s}(\sigma)}:= \widetilde{E}_{\tau(\sigma)}\otimes E_{s}$ over $\widetilde{X}$ 
carries a connection $\nabla^{\tau_{s}(\sigma)}$, defined by the formula 
\begin{equation*}
 \nabla^{\tau_{s}(\sigma)}:=\nabla^{\tau(\sigma)}\otimes 1+1\otimes\nabla.
\end{equation*}
Let $\cdot\colon \mathfrak{p}\otimes S\rightarrow S$ the Clifford multiplication on $\mathfrak{p}\otimes S$.
We extend the Clifford multiplication by requiring that it acts on $V_{\tau_{s}(\sigma)}=S\otimes V_{\tau(\sigma)}$ by
\begin{equation*}
  e\cdot (\phi\otimes \psi)=(e\cdot \phi)\otimes \psi, \quad e\in\Cl(\mathfrak{p}), \phi\in S, \psi\in V_{\tau(\sigma)}.
\end{equation*}
Let $(e_{1},\ldots,e_{d})$ be a local orthonormal 
frame field over an open set $U\subset \widetilde{X}$.
The Dirac operator $\widetilde{D}(\sigma)$ acting on $C^{\infty}(\widetilde{X},\widetilde{E}_{\tau_{s}(\sigma)})$ is defined by
\begin{equation*}
 \widetilde{D} (\sigma)f=\sum_{i=1}^{d}e_{i}\cdot \nabla_{e_{i}}^{\tau_{s}(\sigma)}f.
\end{equation*}
Let $E_{\tau_{s}(\sigma)}:= \Gamma\backslash \widetilde{E}_{\tau_{s}(\sigma)}$ be the locally homogeneous vector bundle over $X$.
Let $\chi:\Gamma\rightarrow \GL(V_{\chi})$ be an arbitrary finite-dimensional representation of $\Gamma$. Let $E_{\chi}$
be the associated flat vector bundle over $X$, equipped with a flat connection $\nabla^{\chi}$.
We consider the product vector bundle $E_{\tau_{s}(\sigma)}\otimes E_{\chi}$ over $X$
and equip this bundle with the product connection $\nabla^{E_{\tau_{s}(\sigma)}\otimes E_{\chi}}$ defined by
\begin{equation*}
 \nabla^{E_{\tau_{s}(\sigma)}\otimes E_{\chi}}:=\nabla^{\tau_{s}(\sigma)}\otimes 1+1\otimes\nabla^{\chi}.
\end{equation*}
We define the Clifford multiplication on $V_{\tau_{s}(\sigma)}\otimes V_{\chi}$ by requiring that it acts only on $V_{\tau_{s}(\sigma)}$, i.e.,
\begin{equation*}
 e\cdot (w\otimes v)=(e\cdot w)\otimes v, \quad e \in\Cl(\mathfrak{p}), w\in V_{\tau_{s}(\sigma)}, v\in V_{\chi}.
\end{equation*}
We consider  an open subset $U$ of $X$ such that $E_{\chi}\lvert_{U}$ is trivial. Let also $(v_{j}), j=1,\dots,m$,  
be any base of flat sections of $E_{\chi}\lvert_{U}$, where $m=\rank(E_\chi)$, and $\phi_{j}\in C^{\infty}(U, E_{\tau_{s}(\sigma)}\lvert_{U})$.
Then,
\begin{equation*}
 E_{\tau_{s}(\sigma)}\otimes E_{\chi}\lvert_{U}\cong\bigoplus_{j=1}^{m}E_{\tau_{s}(\sigma)}\lvert_{U},
\end{equation*}
and for each $\phi \in C^{\infty}(U,E_{\tau_{s}(\sigma)} \otimes E_{\chi}\lvert_{U})$,
\begin{equation*}
\phi=\sum_{j=1}^{m}\phi_{j} \otimes v_{j}.
\end{equation*}
Then,
\begin{equation*}
  \nabla^{E_{\tau_{s}(\sigma)}\otimes E_{\chi}}\phi=\sum_{j=1}^{m}\nabla^{E_{\tau_{s}(\sigma)}}\phi_{j}\otimes v_{j}.
\end{equation*}
The definition above is independent of the choice of the base of flat sections of $E_{\chi}|_{U}$, since the transition maps comparing flat sections are constant.
The twisted Dirac operator $D^{\sharp}_{\chi}(\sigma)$ associated with $\nabla^{E_{\tau_{s}(\sigma)}\otimes E_{\chi}}$
is locally defined by
\begin{equation*}
 D^{\sharp}_{\chi}(\sigma)\phi:=\sum_{i=1}^{d}e_{i}\cdot\nabla_{e_{i}}^{E_{\tau_{s}(\sigma)}\otimes E_{\chi}}\phi.
\end{equation*}
The local definition of the twisted Dirac operator
is independent  of the choice of the orthonormal frame field.
We consider the pullbacks $\widetilde{E}_{\tau_{s}(\sigma)}, \widetilde{E}_{\chi}$ to $\widetilde{X}$ of $E_{\tau_{s}(\sigma)},E_{\chi}$, respectively.
Then, $\widetilde{E}_{\chi}\cong \widetilde{X}\times V_{\chi}$ and
\begin{equation*}
 C^{\infty}(\widetilde{X}, \widetilde{E}_{\tau_{s}(\sigma)}\otimes \widetilde{E}_{\chi})\cong  C^{\infty}(\widetilde{X}, \widetilde{E}_{\tau_{s}(\sigma)})\otimes V_{\chi}.
\end{equation*}
With respect to this isomorphism, it follows from 
the definition of the twisted Dirac operator $D^{\sharp}_{\chi}(\sigma)$
that the lift ${\widetilde{D}}^{\sharp}_{\chi}(\sigma)$ of $D^{\sharp}_{\chi}(\sigma)$ 
to $\widetilde{X}$ is of the form
\begin{equation*}
\widetilde{D}^{\sharp}_{\chi}(\sigma)=\widetilde{D}(\sigma)\otimes \Id_{V_{\chi}}.
\end{equation*}
We recall here the definition of the associated operator $A_{\chi}^{\sharp}(\sigma)$ acting on smooth sections of twisted vector bundles 
(\cite[p. 174-175]{Spilioti2018}). 
Following the proof of Proposition 1.1 in \cite[p. 22]{BO} (see also \cite[Proposition 2.3]{Pf}), 
there exist unique integers
$m_{\tau}(\sigma)\in\{-1,0,1\}$, which are equal to zero except for finitely many $\tau\in \widehat{K}$,
such that,
\begin{itemize}
 \item if $\sigma$ is Weyl invariant, 
$
 \sigma=\sum_{\tau\in\widehat{K}}m_{\tau}(\sigma)i^{*}(\tau);
$
 \item if $\sigma$ is non-Weyl invariant, 
$
 \sigma+w\sigma=\sum_{\tau\in\widehat{K}}m_{\tau}(\sigma)i^{*}(\tau).
$
\end{itemize}
We define a locally homogeneous vector bundle $E(\sigma)$ associated to $\sigma$ in the following way.
\begin{equation*}
 E(\sigma)=\bigoplus_{\substack{\tau\in\widehat{K}\\m_{\tau}(\sigma)\neq 0}}E_{\tau},
\end{equation*}
where $E_{\tau}$ is the locally homogeneous vector bundle over $X$ associated with $\tau\in\widehat{K}$.
By the construction of the locally  homogeneous vector bundles $E(\sigma)$ and $E_{\tau_{s}(\sigma)}$ over $X$ (see \cite[Propos 1.1]{BO}), 
$E(\sigma)=E_{\tau_{s}(\sigma)}$ up to a $\Z_{2}$-grading.

Let $\tau\in\widehat{K}$ and $E_{\tau}$ be the locally homogeneous vector bundle over $X$ associated with $\tau$.
Let $\Delta_{\tau}$ be the Bochner-Laplace operator associated with $\tau$, acting on smooth sections of $E_{\tau}$.
Let $\widetilde{\Delta}_{\tau}$ be the lift of $\Delta_{\tau}$ to $\widetilde{X}$. As in \cite[eq. (5.23)]{Spilioti2018}, we set
$\widetilde{A}_{\tau}:= \widetilde{\Delta}_{\tau}-\lambda_{\tau}\Id$, where $\lambda_{\tau}$ is the Casimir eigenvalue of $\tau$ (\cite[eq. (5.4)]{Spilioti2018}).
Let $\Delta_{\tau,\chi}$ be the twisted 
Bochner-Laplace operator acting on smooth sections of the twisted vector bundle $E_{\tau}\otimes E_{\chi}$.
This operator is defined in \cite[Section 4]{Spilioti2018}. Then, the operator $A_{\tau,\chi}^{\sharp}$ is induced by the twisted Bochner-Laplace operator 
$\Delta^{\sharp}_{\tau,\chi}$ and
acts on smooth sections of the twisted vector bundle
$E_{\tau}\otimes E_{\chi}$. The lift $\widetilde{A}_{\tau,\chi}^{\sharp}$ of $A_{\tau,\chi}^{\sharp}$ to $\widetilde{X}$ is given by
\begin{equation*}
 \widetilde{A}_{\tau,\chi}^{\sharp}=\widetilde{A}_{\tau}\otimes \Id_{V_{\chi}},
\end{equation*}
(\cite[eq. (5.24)]{Spilioti2018}).
Let $\rho$, $\rho_{m}$, $\nu_{\sigma}$ be as in (2.2), (2.3) and (2.5), respectively.
We define the number $c(\sigma)$ by
\begin{equation*}
 c(\sigma):=-\lvert \rho \rvert^{2}-\lvert\rho_{m}\rvert^{2}+\lvert \nu_{\sigma}+\rho_{m}\rvert^{2},
\end{equation*}
and the operator $A_{\chi}^{\sharp}(\sigma)$ by
\begin{equation}
 A_{\chi}^{\sharp}(\sigma):=\bigoplus_{m_{\tau}(\sigma)\neq 0}A^{\sharp}_{\tau,\chi}+c(\sigma).
\end{equation}
(\cite[p. 174-175]{Spilioti2018}).
By \cite[eq. (6.8)]{Spilioti2018}, the Parthasarathy formula generalizes as
\begin{equation*}
D^{\sharp}_{\chi}(\sigma)^{2}=A_{\chi}^{\sharp}(\sigma).
\end{equation*}
The square $D^{\sharp}_{\chi}(\sigma)^2$ of the twisted Dirac operator acting on smooth sections of $E_{\tau_{s}(\sigma)}\otimes E_{\chi}$ is
a second order elliptic differential operator
but no longer self-adjoint. 
Nevertheless, for $x\in X$ and $\xi\in T_{x}^{*}X$, its principal symbol is given by 
\begin{equation*}
 \sigma_{D^{\sharp}_{\chi}(\sigma)^{2}}(x,\xi)=\lVert\xi\rVert^{2}
 \otimes\Id_{(V_{\tau_{s}(\sigma)}\otimes V_{\chi})_{x}}.
\end{equation*}
Hence it has nice spectral properties. By \cite[Lemma 8.6]{Spilioti2018},
its spectrum is discrete
and contained in a translate of a positive cone $C\subset\C$. 
We consider now the corresponding heat operators $e^{-tD^{\sharp}_{\chi}(\sigma)^{2}}$, $D^{\sharp}_{\chi}(\sigma)e^{-tD^{\sharp}_{\chi}(\sigma)^{2}}$.
Let $dg$ a Haar measure on $G$, normalized as in \cite[p. 157]{Spilioti2018}. We equip the spaces $\widetilde{X}$ and $X$
with the induced quotient measure. By \cite[p. 171-173 and p. 181]{Spilioti2018}, $e^{-tD^{\sharp}_{\chi}(\sigma)^{2}}$, $D^{\sharp}_{\chi}(\sigma)e^{-tD^{\sharp}_{\chi}(\sigma)^{2}}$
are well defined, trace class operators, acting on smooth sections of the vector bundle $E_{\tau_{s}(\sigma)}\otimes E_{\chi}$.
\begin{lem}
 The asymptotic expansion of the trace of the kernel $K^{\tau_{s}(\sigma),\chi}_{t}$ of the operator
$D^{\sharp}_{\chi}(\sigma)e^{-tD^{\sharp}_{\chi}(\sigma)^{2}}$ is given by 
\begin{equation}
 \tr K^{\tau_{s}(\sigma),\chi}_{t}(x,x)\sim_{t\rightarrow 0^{+}}\dim(V_{\chi})(a_{0}(\widetilde{x})t^{1/2}+O(t^{3/2},\widetilde{x})),
\end{equation}
where $\widetilde{x}\in \widetilde{X}$ is  the lift of $x$ to $\widetilde{X}$,
$a_{0}(\widetilde{x})$ is a $C^{\infty}$-function on $\widetilde{X}$ and $O(t^{3/2},\widetilde{x}))$ is uniform on $\widetilde{X}$.
\end{lem}
\begin{proof}
Let $K_{t}^{\tau_{s}(\sigma)}$ be the kernel of the the operator $\widetilde{D}(\sigma)e^{-t\widetilde{D}(\sigma)^{2}}$.
Let $x,y\in X$. Let $\widetilde{x},\widetilde{y}\in \widetilde{X}$ be the lifts of $x,y$ to $\widetilde{X}$.
The action of $\Gamma$ on $\widetilde{X}$ induces the following isomorphism of the fibres: 
$R_{\gamma}\colon (\widetilde{E}_{\tau_{s}}(\sigma))_{\widetilde{y}}\rightarrow(\widetilde{E}_{\tau_{s}}(\sigma))_{\gamma\widetilde{y}}$. 
As in  \cite[p. 180--181]{Spilioti2018}, the kernel $K^{\tau_{s}(\sigma),\chi}_{t}$ of the operator $D^{\sharp}_{\chi}(\sigma)e^{-tD^{\sharp}_{\chi}(\sigma)^{2}}$
is given by
\begin{equation}
 K^{\tau_{s}(\sigma),\chi}_{t}(x,y)=\sum_{\gamma\in \Gamma}K_{t}^{\tau_{s}(\sigma)}(\widetilde{x},\gamma\widetilde{y}) \circ(R_{\gamma}\otimes\chi(\gamma)).
\end{equation}
The right-hand side of (4.3) can be written as
\begin{equation}
K_{t}^{\tau_{s}(\sigma)}(\widetilde{x},\widetilde{y})\otimes \Id_{V_{\chi}}+\sum_{\substack{\gamma\in\Gamma\\\gamma\neq e}}
K_{t}^{\tau_{s}(\sigma)}(\widetilde{x},\gamma\widetilde{y}) \circ(R_{\gamma}\otimes\chi(\gamma)).
\end{equation}
Let $\lvert \cdot \rvert$ denote the pointwise norm of the homomorphism  
$K_{t}^{\tau_{s}(\sigma)}\in Hom((\widetilde{E}_{\tau(\sigma)})_{\widetilde{y}},(\widetilde{E}_{\tau(\sigma)})_{\widetilde{x}})$.
Let $d(\cdot,\cdot)$ be the the geodesic distance on $\widetilde{X}$. Recall that the length $l(\gamma)$ of the geodesic associated with $\gamma$ is given by 
$l(\gamma):=\inf\{ d(\widetilde{x},\gamma\widetilde{x})\colon \widetilde{x}\in \widetilde{X}\}$. By \cite[Proposition 3.2]{Muller1998}, 
for every $T>0$, there exists $c>0$ such that
\begin{equation*}
 \lvert K_{t}^{\tau_{s}(\sigma)}(\widetilde{x},\widetilde{y})\rvert \leq c t^{-d/2} e^{-\frac{d^{2}(\widetilde{x},\widetilde{y})}{4t}},
\end{equation*}
for $0<t\leq T$ and $\widetilde{x},\widetilde{y}\in \widetilde{X}$.
Hence,
\begin{align}
 \sum_{\substack{\gamma\in\Gamma\\\gamma\neq e}} \lvert \tr K_{t}^{\tau_{s}(\sigma)}(\widetilde{x},\gamma\widetilde{x})\rvert\notag&\leq
 ct^{-d/2}\sum_{\substack{\gamma\in\Gamma\\\gamma\neq e}}  e^{-\frac{d^{2}(\widetilde{x},\gamma\widetilde{x})}{4t}}\\
 & \leq ct^{-d/2}\sum_{\substack{\gamma\in\Gamma\\\gamma\neq e}}  e^{-\frac{l(\gamma)^{2}}{4t}}.
\end{align}
By the normalization of the Haar measure on $G$ as in \cite[p. 157]{Spilioti2018}, there is a
positive constant $C>0$ such that for every $R>0$ 
\begin{equation*}
 \Vol(B(x_{0},R))\leq Ce^{2|\rho| R},
\end{equation*}
where $\rho$ is as in (2.2).
$\Gamma$ is a cocompact lattice of $G$. This implies that, there exists a positive constant $C'$ such that 
\begin{equation}
\sharp \{[\gamma]:l(\gamma)<R\}\leq\sharp\{ \gamma \in \Gamma: l(\gamma)\leq R\}\leq C'e^{2|\rho |R}
\end{equation}
(\cite[(1.31)]{BO}).
We define 
\begin{equation*}
 \mathcal{N}(R):=\sharp\{[\gamma] \in C( \Gamma)\colon l(\gamma)\leq R\},\quad R\geq 0,
\end{equation*}
where $C(\Gamma)$ denotes the set of $\Gamma$-conjugacy classes.
Since there exists a $c>0$ such that $c\leq l(\gamma)$,
for every $\gamma\neq e$, we have the following estimates.
\begin{align*}
 \sum_{\substack{\gamma\in\Gamma\\\gamma\neq e}}  e^{-\frac{l(\gamma)^{2}}{4t}}&=
 \sum_{k=1}^{\infty}\sum_{\substack{\gamma\neq {e}\\ k c\leq l(\gamma)< (k+1)c}} e^{-\frac{l(\gamma)^{2}}{4t}}\\\notag
 &\leq \sum_{k=1}^{\infty} \mathcal{N}((k+1)c) e^{-\frac{c^{2}k^{2}}{4t}}.
\end{align*}
By (4.6), we get for  $0<t\leq T$,
\begin{equation}
  \sum_{\substack{\gamma\in\Gamma\\\gamma\neq e}}  e^{-\frac{l(\gamma)^{2}}{4t}}
  \leq C'\sum_{k=1}^{\infty} e^{2|\rho |(k+1)c} e^{-\frac{c^{2}k^{2}}{4t}}<\infty.
\end{equation}
By \cite[Lemma 5.1 and Corollary 5.2]{gangolli1968} (see also \cite[p. 28-30]{Wo}), there exists a $c_{0}>0$ such that for every $\gamma\in \Gamma$, 
with $\gamma\neq e$, $l(\gamma)>2c_{0}$. 
Then, by (4.5) 
\begin{align*}
 \sum_{\substack{\gamma\in\Gamma\\\gamma\neq e}} \lvert \tr K_{t}^{\tau_{s}(\sigma)}(\widetilde{x},\gamma\widetilde{x})\rvert
 \leq ct^{-d/2}e^{-c_{0}^{2}/4t}\sum_{\substack{\gamma\in\Gamma\\\gamma\neq e}}  e^{-\frac{(l(\gamma)^{2}-c_{0}^{2})}{4t}},
\end{align*}
and $l(\gamma)^{2}-c_{0}^2>\frac{1}{2}l(\gamma)^{2}$.
Hence, the series on the right-hand side converges and by (4.7) there exists a $C''>0$ such that for all $t$ with $0<t\leq T$,
\begin{align}
 \sum_{\substack{\gamma\in\Gamma\\\gamma\neq e}} \lvert \tr K_{t}^{\tau_{s}(\sigma)}(\widetilde{x},\gamma\widetilde{x})\rvert
 \leq C''t^{-d/2}e^{-c_{0}^{2}/4t}.
\end{align}
We use now Theorem 2.4 in \cite{BF} for the asymptotic expansion of the trace of the kernel $K^{\tau_{s}(\sigma)}_{t}$
as $t\rightarrow 0^{+}$. This theorem is proved for compact manifolds. 
However,  the proof of the existence of the asymptotic expansion is purely local. 
Hence, it holds  for the non-compact case as well.
Therefore, we have the following asymptotic expansion.
\begin{equation}
 \tr{K}^{\tau_{s}(\sigma)}_{t}(\widetilde{x},\widetilde{x})\sim_{t\rightarrow 0^{+}}a_{0}(\widetilde{x})t^{1/2}+O(t^{3/2},\widetilde{x}),
\end{equation}
where $a_{0}(\widetilde{x})$ is a smooth function determined by the total symbol 
of $\widetilde{D}(\sigma)$
 and $O(t^{3/2},\widetilde{x}))$ is uniform in $\widetilde{x}\in\widetilde{X}$.
Hence, by (4.3), (4.4), (4.8) and (4.9) we get
\begin{equation*}
  \tr K^{\tau_{s}(\sigma),\chi}_{t}(x,x)\sim_{t\rightarrow 0^{+}}\dim(V_{\chi})(a_{0}(\widetilde{x})t^{1/2}+O(t^{3/2},\widetilde{x})).
\end{equation*}
\end{proof}
\textbf{Remark:} We mention here that in Proposition 3.2 in \cite{Muller1998}, the representation of $K$
is irreducible. However, the irreducibility condition is not used in the proof.
Hence, this result can be extended to the case of the non-irreducible representation
$\tau_{s}(\sigma)$ of $K$.
\newline
\newline
We consider now the operator $e^{-t(A^{\sharp}_{\tau,\chi}+c(\sigma))}$ induced by each summand $A^{\sharp}_{\tau,\chi}+c(\sigma)$
in the definition (4.1) of $A^{\sharp}_{\chi}(\sigma)$.
\begin{lem}
There exist coefficients $c_{j}$ such that 
the asymptotic expansion of the trace of the kernel $H^{\tau,\chi}_{t}(x,y)$ of the operator
$e^{-t(A^{\sharp}_{\tau,\chi}+c(\sigma))}$  is given by 
\begin{equation}
\tr H^{\tau,\chi}_{t}(x,x)\sim_{t\rightarrow 0^{+}}\dim(V_{\chi})\sum_{j=0}^{\infty}c_{j}t^{\frac{j-d}{2}}.
\end{equation}
\end{lem}
\begin{proof}
Let $Q^{\tau}_{t}$
be the kernel associated with the operator
$e^{-t\widetilde{A}_{\tau}}$ (see \cite[p. 175-176]{Spilioti2018}).
Let $x,y\in X$. Let $\widetilde{x},\widetilde{y}\in \widetilde{X}$ be the lifts of $x,y$ to $\widetilde{X}$, respectively.
Let $R_{\gamma}$ be the following isomorphism:
$R_{\gamma}\colon (\widetilde{E}_{\tau})_{\widetilde{y}}\rightarrow(\widetilde{E}_{\tau})_{\gamma\widetilde{y}}$.
As in \cite[p. 175]{Spilioti2018},
the kernel $H^{\tau,\chi}_{t}$ of the operator  $e^{-t(A^{\sharp}_{\tau,\chi}+c(\sigma))}$
is given by
\begin{equation*}
H^{\tau,\chi}_{t}(x,y)= \sum_{\gamma\in\Gamma}e^{-tc(\sigma)}Q^{\tau}_{t}(\widetilde{x},\gamma\widetilde{y}) \circ(R_{\gamma}\otimes\chi(\gamma)).
\end{equation*}
Equivalently,
\begin{equation}
H^{\tau,\chi}_{t}(x,y)= e^{-tc(\sigma)}Q^{\tau}_{t}(\widetilde{x},\widetilde{y})\otimes \Id_{V_{\chi}}+
\sum_{\substack{\gamma\in\Gamma\\\gamma\neq e}}e^{-tc(\sigma)}
Q^{\tau}_{t}(\widetilde{x},\gamma\widetilde{y}) \circ(R_{\gamma}\otimes\chi(\gamma)).
\end{equation}
As in the proof of Lemma 4.1, one can prove that there exists a $c>0$ such that for $t\rightarrow 0^{+}$,
\begin{equation}
\sum_{\substack{\gamma\in\Gamma\\\gamma\neq e}}e^{-tc(\sigma)}
\lvert \tr Q^{\tau}_{t}(\widetilde{x},\gamma\widetilde{x})\rvert\leq ct^{-d/2}e^{-c_{0}^{2}/4t}e^{-tc(\sigma)}.
\end{equation}
By \cite[Lemma 1.7.4]{Gilk} the trace  of the kernel ${Q}^{\tau}_{t}$,
associated with the operator $e^{-t\widetilde{A}_{\tau}}$ has the asymptotic expansion
\begin{equation*}
 \tr{Q}^{\tau}_{t}(\widetilde{x},\widetilde{x})\sim_{t\rightarrow 0^{+}}\sum_{j=0}^{\infty}c_{j}(\widetilde{x})t^{\frac{j-d}{2}}.
\end{equation*}
where $c_{j}(\widetilde{x})$ are smooth functions determined by the total symbol of the operator $\widetilde{\Delta}_{\tau}$.
The proof of the existence of the asymptotic expansion is purely local. 
Therefore, it applies in our case as well.
Since $e^{-t\widetilde{A}_\tau}$ commutes with the action of $G$, and $G$ acts
transitively on $\widetilde{X}$, it follows that $Q_{t}^{\tau}(\widetilde {x},\widetilde {x})$ is independent of $\widetilde{x}$. 
Therefore, the coefficients $c_{j}(\widetilde{x})$
are also independent of $\widetilde{x}$.
Let $\widetilde{x}_{0}=eK\in\widetilde{X}$
be the base point.
Then, we have
\begin{equation}
\tr{Q}^{\tau}_{t}(\widetilde{x}_{0},\widetilde{x}_{0})\sim_{t\rightarrow 0^{+}}\sum_{j=0}^{\infty}c_{j}t^{\frac{j-d}{2}}.
\end{equation}
Moreover, one can use the expansion in power series of the term $e^{-tc(\sigma)}$.
By (4.11), (4.12) and (4.13), we get
\begin{equation*}
\tr H^{\tau,\chi}_{t}(x,x)\sim_{t\rightarrow 0^{+}}\dim(V_{\chi})\sum_{j=0}^{\infty}c_{j}
t^{\frac{j-d}{2}}
\end{equation*}
\end{proof}
We need the following definitions.
\begin{defi}
Let  $R_{\theta}:=\{\rho e^ {i\theta}: \rho\in[0,\infty]\}$. 
The angle $\theta\in[0,2\pi)$ is a principal angle for the elliptic operator $D^{\sharp}_{\chi}(\sigma)$ if 
\begin{equation*}
 \spec(\sigma_{D^{\sharp}_{\chi}(\sigma)}(x,\xi))\cap R_{\theta}=\emptyset,\quad \forall x\in X,\forall\xi\in T_{x}^{*}X,\xi\neq 0.
\end{equation*}
\end{defi}
\begin{defi}
Let $I\subset [0,2\pi)$.
Let $L_{I}$ be a solid angle defined by
\begin{equation*}
 L_{I}:=\{\rho e^ {i\theta}: \rho\in(0,\infty),\theta\in I\}.
\end{equation*}
The angle $\theta$ is an Agmon angle for the elliptic operator  $D^{\sharp}_{\chi}(\sigma)$, if it is a principal angle for $D^{\sharp}_{\chi}(\sigma)$ 
and there exists an $\varepsilon>0$ such that 
\begin{equation*}
 \spec(D^{\sharp}_{\chi}(\sigma))\cap L_{[\theta-\varepsilon,\theta+\varepsilon]}=\emptyset.
\end{equation*}
\end{defi}
We define here the eta function associated with non-self-adjoint operators with elliptic, self-adjoint principal symbol (see \cite{Gilkeysoviet}).

\begin{defi}\textbf{Eta function of $D^{\sharp}_{\chi}(\sigma)$}.
The principal symbol of $D^{\sharp}_{\chi}(\sigma)$
is self-adjoint. Hence, the angles $\pm \pi/2$ are principal angles for $D^{\sharp}_{\chi}(\sigma)$.
Since $D^{\sharp}_{\chi}(\sigma)$ possesses a principal angle, it also possesses an 
Agmon angle (see \cite[Section 3.10]{BK2}).
Let $\theta$ be an Agmon angle for $D^{\sharp}_{\chi}(\sigma)$.
Let $\spec(D^{\sharp}_{\chi}(\sigma))=\{\lambda_{k}\colon k\in \N\}$ be the spectrum of $D^{\sharp}_{\chi}(\sigma)$. It is a discrete subset of $\C$.

By \cite[\S I.6]{Mk}, the space $L^{2}(X,E_{\tau_{s}(\sigma)}\otimes E_{\chi})$ of square integrable sections 
of $E_{\tau_{s}(\sigma)}\otimes E_{\chi}$ is the closure of the algebraic direct sum of finite dimensional $D^{\sharp}_{\chi}(\sigma)$-invariant subspaces
\begin{equation}
L^{2}(X,E_{\tau_{s}(\sigma)}\otimes E_{\chi})=\overline{\bigoplus\limits_{k\geq 1} \Lambda_{k}},
\end{equation}
such that the restriction of $D^{\sharp}_{\chi}(\sigma)$ to $\Lambda_{k}$
has a unique eigenvalue $\lambda_{k}$
and $\lim_{k\rightarrow \infty}\vert \lambda_{k}\vert=\infty$.
In general, the sum (4.14) is not a sum of mutually orthogonal subspaces. 
The spaces $\Lambda_{k}$ are called the root vectors if $D^{\sharp}_{\chi}(\sigma)$ with eigenvalue 
$\lambda_{k}$.  The algebraic multiplicity $m_{k}$ of the the eigenvalue $\lambda_{k}$ is defined as the 
the dimension of the space $\Lambda_{k}$.

Denote by $\log_{\theta}\lambda_{k}$ the branch of the logarithm in $\C\textbackslash  R_{\theta}$ with 
$\theta<\text{Im}(\log_{\theta}\lambda_{k})<\theta+2\pi$.
Let $(\lambda_{k})_{\theta}:=e^{\log_{\theta}\lambda_{k}}$.
For Re$(s)\gg0$, we define the eta function $ \eta_{\theta}(s,D^{\sharp}_{\chi}(\sigma))$ of $D^{\sharp}_{\chi}(\sigma)$ by 
\begin{equation*}
 \eta_{\theta}(s,D^{\sharp}_{\chi}(\sigma)):=\sum_{\text{\Re}(\lambda_{k})>0} m_{k}(\lambda_{k})_{\theta}^{-s}-\sum_{\text{\Re}(\lambda_{k})<0} m_{k}(-\lambda_{k})_{\theta}^{-s}.
\end{equation*}
\end{defi}
Note that since the angles $\pm \pi/2$ are principal angles for $D^{\sharp}_{\chi}(\sigma)$, there are at most finitely many eigenvalues of $D^{\sharp}_{\chi}(\sigma)$ on or near the imaginary axis.
Hence, the eigenvalues  of $D^{\sharp}_{\chi}(\sigma)$ that are purely imaginary do not
contribute to the definition of the eta function. 

It has been shown by Grubb and Seeley (\cite[Theorem 2.7]{GS}) that $\eta_{\theta}(s,D^{\sharp}_{\chi}(\sigma))$ has a meromorphic continuation to the whole complex plane $\C$ with isolated simple poles and that
is regular at $s=0$. Moreover, the number $\eta_{\theta}(0,D^{\sharp}_{\chi}(\sigma))$ is independent of the Agmon angle $\theta$.
Hence, we write $\eta(0,D^{\sharp}_{\chi}(\sigma))$ instead of $\eta_{\theta}(0,D^{\sharp}_{\chi}(\sigma))$.
We give here a short description of the proof.

Let $\theta$ be an Agmon angle for $D^{\sharp}_{\chi}(\sigma)$. We assume that $0$ is not an
eigenvalue of $D^{\sharp}_{\chi}(\sigma)$.
To define the operator $D^{\sharp}_{\chi}(\sigma)^{-s}$, one has to use the contour $\Gamma_{\theta,\rho_{0}}$, described as in  \cite[p. 88]{Sh}. 
There exists a $\rho_{0}>0$ such that
\begin{equation*}
\spec(D^{\sharp}_{\chi}(\sigma))\cap\{z\in\C:\lvert z \rvert\leq 2\rho_{0}\}=\emptyset.
\end{equation*}
We consider the contour $\Gamma_{\theta,\rho_{0}}\subset\C$, defined as $\Gamma_{\theta,\rho_{0}}=\Gamma_{1}\cup\Gamma_{2}\cup\Gamma_{3}$, where 
$\Gamma_{1}=\{re^{i\theta}:\infty>r\geq \rho_{0}\}$, $\Gamma_{2}=\{\rho_{0}e^{i\beta}:\theta\leq \beta\leq \theta-2\pi\}$,
$\Gamma_{3}=\{re^{i(\theta-2\pi)}:\rho_{0}\leq r<\infty\}$. On $\Gamma_{1}$, $r$ runs from $\infty$ to $\rho_{0}$, $\Gamma_{2}$ is oriented counterclockwise, and on $\Gamma_{3}$,
$r$ runs from $\rho_{0}$ to $\infty$.
Then, for Re$(s)>0$ we define
\begin{equation*}
D^{\sharp}_{\chi}(\sigma)^{-s}:=\frac{i}{2\pi}\int_{\Gamma_{\theta,\rho_{0}}}\lambda^{-s}\big(D^{\sharp}_{\chi}(\sigma)-\lambda\Id\big)^{-1} d\lambda.
\end{equation*}
Let $\Pi_{>}$ (resp. $\Pi_{<}$) be the pseudo-differential projection whose image contains the span of all generalized eigenvectors of $D^{\sharp}_{\chi}(\sigma)$
corresponding to eigenvalues $\lambda$ with Re$(\lambda)>0$ (resp. Re$(\lambda)$<0) (for more details see \cite[Definition 6.16]{BK3}).
The zeta function $\zeta_{\theta}(s,P,D^{\sharp}_{\chi}(\sigma))$ is define for Re$(s)>d$ by,
\begin{equation*}
 \zeta_{\theta}(s,P,D^{\sharp}_{\chi}(\sigma)):=\Tr(P D^{\sharp}_{\chi}(\sigma)^{-s}),
\end{equation*}
where $P=\Pi_{>},\Pi_{<}$ (\cite[p. 38]{BK3}). 
Then, Definition 4.5 can be read as
\begin{equation*}
 \eta_{\theta}(s,D^{\sharp}_{\chi}(\sigma))=\zeta_{\theta}(s,\Pi_{>},D^{\sharp}_{\chi}(\sigma))-\zeta_{\theta}(s,\Pi_{<},D^{\sharp}_{\chi}(\sigma)).
\end{equation*}
The zeta function $\zeta_{\theta}(s,P,D^{\sharp}_{\chi}(\sigma))$ is define for Re$(s)>d$ by,
\begin{equation*}
 \zeta_{\theta}(s,P,D^{\sharp}_{\chi}(\sigma)):=\Tr(P D^{\sharp}_{\chi}(\sigma)^{-s}).
\end{equation*}
If we integrate by parts the integral above, the operator $(D^{\sharp}_{\chi}(\sigma)-\lambda\Id)^{-k}$
will occur. By (\cite[Theorem 2.7.]{GS}), for $k>d$, there exists an asymptotic expansion 
of the trace of the operator $P(D^{\sharp}_{\chi}(\sigma)-\lambda\Id)^{-k}$ as $|\lambda|\rightarrow\infty$:
\begin{equation*}
 \Tr(P(D^{\sharp}_{\chi}(\sigma)-\lambda\Id)^{-k})\sim\sum_{j=1}^{\infty}c_{j}\lambda^{d-j-k}+\sum_{l=1}^{\infty}(c_{l}'\log\lambda+c_{l}'')\lambda^{-k-l},
\end{equation*}
where the coefficients $c_{j}$ and $c_{l}'$ are determined from the symbols of $D^{\sharp}_{\chi}(\sigma)$ and $P$,
and the coefficients $c_{l}''$ are in general globally determined.

 Let $m_{+}$, respectively $m_{-}$, denote the the number of the eigenvalues of $D^{\sharp}_{\chi}(\sigma)$
on the positive, respectively negative, part of the imaginary axis.
We define the eta invariant $\eta(D^{\sharp}_{\chi}(\sigma))$ of the operator $D^{\sharp}_{\chi}(\sigma)$ by 
\begin{equation*}
\eta(D^{\sharp}_{\chi}(\sigma))=\frac{\eta(0,D^{\sharp}_{\chi}(\sigma))+m_{+}-m_{-}}{2}.
\end{equation*}

Let  now $\Pi_{+}$ be the projection on the span of the root spaces corresponding to eigenvalues $\lambda$ with Re$(\lambda_{k}^{2})>0$.
We define the functions
\begin{align*}
\eta_{0}(s,D^{\sharp}_{\chi}(\sigma)):=&\sum_{\substack{\text{Re}(\lambda_{k})>0\\ \text{Re}(\lambda_{k}^{2})\leq0}}m_{k}\lambda_{k}^{-s}-
\sum_{\substack{\text{Re}(\lambda_{k})<0\\ \text{Re}(\lambda_{k}^{2})\leq 0}}m_{k}\lambda_{k}^{-s}\\
\eta_{1}(s,D^{\sharp}_{\chi}(\sigma)):=&\sum_{\substack{\text{Re}(\lambda_{k})>0\\ \text{Re}(\lambda_{k}^{2})>0}}m_{k}\lambda_{k}^{-s}-
\sum_{\substack{\text{Re}(\lambda_{k})<0\\ \text{Re}(\lambda_{k}^{2})>0}}m_{k}\lambda_{k}^{-s}.\\
\end{align*}
By Definition 4.5, the eta function $\eta(s,D^{\sharp}_{\chi}(\sigma))$ satisfies the equation
\begin{equation*}
 \eta(s,D^{\sharp}_{\chi}(\sigma))=\eta_{0}(s,D^{\sharp}_{\chi}(\sigma))+\eta_{1}(s,D^{\sharp}_{\chi}(\sigma)).
\end{equation*}
Since the spectrum of $D^{\sharp}_{\chi}(\sigma)^{2}$ is discrete and contained in a translate of a positive cone in $\C$,
there are only finitely many eigenvalues of $D^{\sharp}_{\chi}(\sigma)$ with Re$(\lambda_{k}^{2})\leq 0$. 

\begin{lem}
 The eta function $\eta(s,D^{\sharp}_{\chi}(\sigma))$ satisfies the equation
 \begin{equation}
 \eta(s,D^{\sharp}_{\chi}(\sigma))= \eta_{0}(s,D^{\sharp}_{\chi}(\sigma))+\frac{1}
 {\Gamma(\frac{s+1}{2})}\int_{0}^{\infty}\Tr(\Pi_{+}D^{\sharp}_{\chi}(\sigma) e^{-tD^{\sharp}_{\chi}(\sigma)^{2}})t^{\frac{s-1}{2}}dt.
\end{equation}
\end{lem}
\begin{proof} 
Let $\lambda_{k}$ be an eigenvalue of $D^{\sharp}_{\chi}(\sigma)$ such that Re$(\lambda_{k}^{2})>0$.
We have
\begin{equation*}
 (\lambda_{k}^{2})^{-\frac{s+1}{2}}=\frac{1}{\Gamma(\frac{s+1}{2})}\int_{0}^{\infty}e^{-\lambda_{k}^{2}t}t^{\frac{s-1}{2}}dt.
\end{equation*}
We mention here that we can use the Lidskii's theorem (\cite[Theorem 3.7, p. 35]{SB}) to express the trace of the operator 
$D^{\sharp}_{\chi}(\sigma) e^{-tD^{\sharp}_{\chi}(\sigma)^{2}}$ in terms of its eigenvalues $\lambda_{k}$
\begin{equation*}
 \Tr(D^{\sharp}_{\chi}(\sigma) e^{-tD^{\sharp}_{\chi}(\sigma)^{2}})=\sum_{\lambda_{k}\neq 0}m_{k}\lambda_{k}e^{-t\lambda_{k}^{2}}.
\end{equation*}
Taking the sum over the eigenvalues $\lambda_{k}$ of $D^{\sharp}_{\chi}(\sigma)$, counting also their algebraic multiplicities, we have
\begin{equation}
 \Tr (\Pi_{+}D^{\sharp}_{\chi}(\sigma)(D^{\sharp}_{\chi}(\sigma)^{2})^{-\frac{s+1}{2}})=
 \frac{1}{\Gamma(\frac{s+1}{2})}\int_{0}^{\infty}\Tr(\Pi_{+}D^{\sharp}_{\chi}(\sigma) e^{-tD^{\sharp}_{\chi}(\sigma)^{2}})t^{\frac{s-1}{2}}dt.
\end{equation}
To prove the convergence of the above integral, we first observe that 
\begin{align}
\Tr (\Pi_{+}D^{\sharp}_{\chi}(\sigma)(D^{\sharp}_{\chi}(\sigma)^{2})^{-\frac{s+1}{2}})\notag&=\int_{0}^{1}\Tr(\Pi_{+}D^{\sharp}_{\chi}(\sigma) 
e^{-tD^{\sharp}_{\chi}(\sigma)^{2}})t^{\frac{s-1}{2}}dt+\\
&\int_{1}^{\infty}\Tr(\Pi_{+}D^{\sharp}_{\chi}(\sigma) e^{-tD^{\sharp}_{\chi}(\sigma)^{2}})t^{\frac{s-1}{2}}dt.
\end{align}
For the first integral in the right-hand side of (4.17),
we use the asymptotic expansion of the trace of the operator $D^{\sharp}_{\chi}(\sigma) e^{-tD^{\sharp}_{\chi}(\sigma)^{2}}$.
By Lemma 4.1, we have
\begin{align}
\notag\int_{0}^{1}\Tr(\Pi_{+}D^{\sharp}_{\chi}(\sigma) e^{-tD^{\sharp}_{\chi}(\sigma)^{2}})t^{\frac{s-1}{2}}dt=
&\int_{0}^{1}\dim (V_{\chi})( \alpha_{0}t^{1/2}+O(t^{3/2}))t^{\frac{s-1}{2}}dt\\
&=\dim (V_{\chi}) \alpha_{0}\frac{2}{s+2}+\int_{0}^{1}O(t^{3/2})t^{\frac{s-1}{2}}dt,
\end{align}
which is a holomorphic function for $\text{\Re}(s)>-2$.
Here,
\begin{equation*}
 \alpha_{0}=\int_{X}a_{0}(x)d\mu(x),
\end{equation*}
where $a_{0}(x)$ is a smooth function, and $\mu(x)$ is the volume measure determined by the Riemannian metric on $X$.

We treat now the second integral in the right-hand side of (4.17).
Since there are finitely many eigenvalues $\lambda_{k}$
with $\vert\lambda_{k}\vert<1$, we have for $t\geq 1$,
\begin{equation}
\Big|\sum_{\text{Re}(\lambda_{k}^{2})>0}m_{k} \lambda_{k} e^{-t\lambda_{k}^{2}}\Big|
 \leq C \sum_{\text{Re}(\lambda_{k}^{2})>0} m(\lambda_{k}^{2})\vert\lambda_{k}^{2} \vert e^{-t\text{Re}(\lambda_{k}^{2})},
\end{equation}
where $C$ is a positive constant.
We recall here  that the spectrum of $D^{\sharp}_{\chi}(\sigma)^{2}$ is discrete and contained 
in a translate of a positive cone in $\C$.
We set $c_{0}:=\frac{1}{2}\min\{\text{Re}(\lambda_{k}^{2})\colon\text{Re}(\lambda_{k}^{2})>0\}$. Then, we have for $t\geq 1$,
\begin{align}
\sum_{\text{Re}(\lambda_{k}^{2})>0} m(\lambda_{k}^{2})\vert\lambda_{k}^{2} \vert e^{-t\text{Re}(\lambda_{k}^{2})}
 \leq  e^{-c_{0}t/2}  \sum_{\text{Re}(\lambda_{k}^{2})>0} m(\lambda_{k}^{2})\vert\lambda_{k}^{2} \vert e^{-\text{Re}(\lambda_{k}^{2})/2}.
\end{align}
We use now Weyl's law  for the non-self-adjoint operator 
$D^{\sharp}_{\chi}(\sigma)^{2}$ to estimate the last sum. 
Let $r$ be a positive constant. 
We define the counting function $\mathcal{N}(r)$ by
\begin{equation*}
 \mathcal{N}(r):=\sum_{\substack{\lambda^{2}\in\spec(D^{\sharp}_{\chi}(\sigma)^{2})\\{|\lambda^{2}|\leq r}}} m(\lambda^{2}).
\end{equation*}
In \cite{M1}, the generalization of  Weyl's law for the non-self-adjoint case is proved.
By \cite[Lemma 2.2]{M1}, we have
\begin{equation*}
 \mathcal{N}(r)=\frac{\rank(\rank(E(\tau_{s}(\sigma)\otimes E_{\chi}))\Vol(X)}{(4\pi)^{d/2}\Gamma(d/2+1)}r^{d/2}+o(r^{d/2}), \quad r\rightarrow \infty,
\end{equation*}
where $\rank(E(\tau_{s}(\sigma)\otimes E_{\chi})$ denotes the rank of the product vector bundle $E(\tau_{s}(\sigma))\otimes E_{\chi}$. 
Let  $a>0$ be the slope of the boundary of the cone, in which all the eigenvalues $\lambda_{j}^{2}$ of $D^{\sharp}_{\chi}(\sigma)^{2}$ are contained. 
We have
\begin{equation*}
 \sharp\{j\colon|\text{Re}(\lambda_{j}^{2})| \leq r\}  \leq \sharp\{j\colon|\lambda_{j}^{2}|\leq \sqrt{1+a^{2}} r\}\leq \mathcal{N}(\sqrt{1+a^{2}}r).
\end{equation*}
Hence, we have
\begin{align}
 \sum_{\text{Re}(\lambda_{k}^{2})>0} m(\lambda_{k}^{2})\vert\lambda_{k}^{2} \vert e^{-\text{Re}(\lambda_{k}^{2})/2}
 < \infty.
\end{align}
By (4.19), (4.20) and (4.21), there exists a positive constant $c$ such that 
\begin{equation*}
 \Big|\sum_{\text{Re}(\lambda_{k}^{2})>0}m_{k} \lambda_{k} e^{-t\lambda_{k}^{2}}\Big|\leq 
 ce^{-c_{0}t/2}.
\end{equation*}
Therefore,
\begin{equation}
 \int_{1}^{\infty} \lvert\Tr(\Pi_{+}D^{\sharp}_{\chi}(\sigma) e^{-tD^{\sharp}_{\chi}(\sigma)^{2}})
 t^{\frac{s-1}{2}}\rvert dt\leq c\int_{1}^{\infty}e^{\frac{-c_{0}t}{2}} t^{\text{\Re}(\frac{s}{2})-\frac{1}{2}}dt<\infty.
\end{equation}
By (4.17), (4.18), and (4.22), it follows that for $\text{\Re}(s)>-2$,
\begin{equation*}
 \eta(s,D^{\sharp}_{\chi}(\sigma))=\eta_{0}(s,D^{\sharp}_{\chi}(\sigma))+\frac{1}{\Gamma(\frac{s+1}{2})}\int_{0}^{\infty}
 \Tr(\Pi_{+}D^{\sharp}_{\chi}(\sigma) e^{-tD^{\sharp}_{\chi}(\sigma)^{2}})t^{\frac{s-1}{2}}dt.
\end{equation*}
\end{proof}

\section{Functional equations for the Selberg zeta function}

The meromorphic continuation of the Selberg and Ruelle zeta function 
associated with non-unitary representations of the subgroup 
$\Gamma$ is proved in \cite{Spilioti2018}.
Our aim is to derive the functional equations for the Selberg zeta function in case (a), and for the symmetrized and super zeta function in case (b). We recall here the formulas proved in \cite{Spilioti2018} (see (5.1) and (5.2) below), using the generalized resolvent identity and the trace formulas. 
More precisely, let $N\in\N$ and $i=1,\ldots,N$. Let $s_{i}$ be complex numbers with $s_{i}^2\neq s_{j}^2$ for all $i\neq j$. Let $\text{\Re}(s_{i}^2) \gg 0$.
We have the following expressions
\begin{align*}
R(s_{i}^{2})=(A_{\chi}^{\sharp}(\sigma)+s_{i}^{2})^{-1}&=\int_{0}^{\infty}e^{-ts_{i}^{2}}e^{-t{A_{\chi}^{\sharp}(\sigma)}}dt;\\
 D^{\sharp}_{\chi}(\sigma)R(s_{i}^{2})=D^{\sharp}_{\chi}(\sigma){{(D^{\sharp}_{\chi}(\sigma)}^{2}+s_{i}^{2})}^{-1}&=\int_{0}^{\infty}e^{-ts_{i}^{2}}D^{\sharp}_{\chi}(\sigma)
 e^{-t{D^{\sharp}_{\chi}(\sigma)}^{2}}dt\\
\end{align*}
(recall that ${D^{\sharp}_{\chi}(\sigma)}^{2}=A_{\chi}^{\sharp}(\sigma)$).
We choose $N$ such that $N>\frac{d}{2}+1$. Then, the operators 
$D^{\sharp}_{\chi}(\sigma)R(s_{i}^{2})$ and $R(s_{i}^{2})$ are trace class.
By \cite[Lemma 7.5]{Spilioti2018}, for $N>\frac{d}{2}$, the operator $\prod_{i=1}^{N}R(s_{i}^{2})$ is trace class,
and for $N>\frac{d}{2}+1$, the operator $D^{\sharp}_{\chi}(\sigma)\prod_{i=1}^{N}R(s_{i}^{2})$ is trace class.
We use the generalized resolvent identity as in Lemma 3.5 in  \cite{BO}. In our case this lemma reads
\begin{align*}
 \prod_{i=1}^{N}R(s_{i}^{2})&=
\sum_{i=1}^{N}\bigg(\prod_{\substack{j=1\\ j\neq i}}^{N}\frac{1}{s_{j}^{2}-s_{i}^{2}}\bigg)R(s_{i}^{2});\\
D^{\sharp}_{\chi}(\sigma)  \prod_{i=1}^{N}R(s_{i}^{2})&=D^{\sharp}_{\chi}(\sigma)\sum_{i=1}^{N}\bigg(\prod_{\substack{j=1\\ j\neq i}}^{N}
\frac{1}{s_{j}^{2}-s_{i}^{2}}\bigg)R(s_{i}^{2}).
\end{align*}
Hence, if we take into account the integral representations of the resolvent operators above, we have
\begin{align*}
  \prod_{i=1}^{N}(A_{\chi}^{\sharp}(\sigma)+s_{i}^{2})^{-1}&=
\int_{0}^{\infty}\sum_{i=1}^{N}\bigg(\prod_{\substack{j=1\\ j\neq i}}^{N}\frac{1}{s_{j}^{2}-s_{i}^{2}}\bigg)e^{-ts_{i}^{2}}e^{-t{A_{\chi}^{\sharp}(\sigma)}}dt;\\
 D^{\sharp}_{\chi}(\sigma)\prod_{i=1}^{N}({D^{\sharp}_{\chi}(\sigma)}^{2}+s_{i}^{2})^{-1}&=\int_{0}^{\infty}\sum_{i=1}^{N}\bigg(\prod_{\substack{j=1\\ j\neq i}}^{N}\frac{1}{s_{j}^{2}-s_{i}^{2}}\bigg)e^{-ts_{i}^{2}}
  D^{\sharp}_{\chi}(\sigma)e^{-t{D^{\sharp}_{\chi}(\sigma)}^{2}}dt.
\end{align*}
We plug now the trace of the corresponding operators and use the trace formulas  \cite[eq. (7.27) and (7.34)]{Spilioti2018}). 
\newline
\newline
\textbf{Remark:} In the trace formulas (7.27) and (7.34) in \cite{Spilioti2018},
the traces of the operators are super traces, corresponding to the grading of the locally homogeneous vector bundle $E(\sigma)$ over $X$, defined as in Section 4. For more details about the grading of  
$E(\sigma)$, we refer the reader to \cite[p. 27, 29]{BO} and \cite[p. 175]{Spilioti2018}).
We denote the super trace of the corresponding operators by $\Tr_{s}$.
By \cite[eq. (7.27) and (7.34)]{Spilioti2018}), we have
\begin{itemize}
 \item \textbf{case (a)}
 \begin{align}
\Tr _{s}\prod_{i=1}^{N}(A_{\chi}^{\sharp}(\sigma)+s_{i}^{2})^{-1}\notag&=\sum_{i=1}^{N}\bigg(\prod_{\substack{j=1\\ j\neq i}}^{N}\frac{1}{s_{j}^{2}-s_{i}^{2}}\bigg)\frac{\pi}{s_{i}} \dim(V_{\chi})\Vol(X)P_{\sigma}(s_{i})\\
&+\sum_{i=1}^{N}\frac{1}{2s_{i}}\bigg(\prod_{\substack{j=1\\ j\neq i}}^{N}\frac{1}{s_{j}^{2}-s_{i}^{2}}\bigg)L(s_{i}).
 \end{align}
\item \textbf{case (b)}
\begin{align}
 \label{eq mer sym}
\Tr _{s}\prod_{i=1}^{N}(A_{\chi}^{\sharp}(\sigma)+s_{i}^{2})^{-1}\notag&=\sum_{i=1}^{N}\bigg(\prod_{\substack{j=1\\ j\neq i}}^{N}\frac{1}{s_{j}^{2}-s_{i}^{2}}\bigg)\frac{\pi}{s_{i}} 2\dim(V_{\chi})\Vol(X)P_{\sigma}(s_{i})\\
 &+\sum_{i=1}^{N}\bigg(\prod_{\substack{j=1\\ j\neq i}}^{N}\frac{1}{s_{j}^{2}-s_{i}^{2}}\bigg)\frac{1}{2s_{i}}L_{S}(s_{i}).
 \end{align}
\end{itemize}

We consider first case (a).
\begin{lem}
 The logarithmic derivative $L(s)$ of the Selberg zeta function 
 satisfies the following functional equation
\begin{equation}
L(s)+L(-s)=-4\pi \dim(V_{\chi}) \Vol(X) P_{\sigma}(s).
\end{equation}
\end{lem}
\begin{proof}
We fix the complex numbers $s_{2},\ldots,s_{N}\in\C$ and let $s_{1}=s\in\C$ vary. 
The left-hand side of (5.1) is invariant under $s\mapsto -s$.
Hence, the right-hand side of (5.1) is also invariant under $s\mapsto -s$.
We have
\begin{align*}
&\bigg(\prod_{\substack{j=2\\ j\neq i}}^{N}\frac{1}{s_{j}^{2}-s^{2}}\bigg)\frac{1}{2s}L(s)\mapsto \bigg(\prod_{\substack{j=2\\ j\neq i}}^{N}\frac{1}{s_{j}^{2}-s^{2}}\bigg)\frac{1}{-2s}L(-s).\\
\end{align*}
Since the Plancherel polynomial
is an even polynomial of $s$ 
\begin{align*}
-\bigg(\prod_{\substack{j=2\\ j\neq i}}^{N}\frac{1}{s_{j}^{2}-s^{2}}\bigg)\frac{\pi}{s}\dim(V_{\chi})\Vol(X)P_{\sigma}(s)\mapsto
 \bigg(\prod_{\substack{j=2\\ j\neq i}}^{N}\frac{1}{s_{j}^{2}-s^{2}}\bigg)\frac{\pi}{s}\dim(V_{\chi})\Vol(X)P_{\sigma}(s).\\\notag
\end{align*}
We subtract the resulting equation from (5.1).
Hence,
\begin{equation*}
\bigg(\prod_{\substack{j=2\\ j\neq i}}^{N}\frac{1}{s_{j}^{2}-s^{2}}\bigg)\frac{1}{2s}(L(s)+L(-s))
=-\bigg(\prod_{\substack{j=2\\ j\neq i}}^{N}\frac{1}{s_{j}^{2}-s^{2}}\bigg)\frac{2\pi}{s}\dim(V_{\chi})\Vol(X)P_{\sigma}(s).
\end{equation*}
We multiply the above equation by the function 
$
 2s\prod_{j=2}^{N}(s_{j}^{2}-s^{2}).
$
Then, we get
\begin{equation*}
L(s)+L(-s)=-4\pi\dim(V_{\chi})\Vol(X)P_{\sigma}(s).
\end{equation*}
\end{proof}
\begin{thm}
 The Selberg zeta function $Z(s;\sigma,\chi)$ satisfies the following functional equation
\begin{equation}
 \frac{Z(s;\sigma,\chi)}{Z(-s;\sigma,\chi)}=\exp\bigg(-4\pi\dim(V_{\chi})\Vol(X)\int_{0}^{s}P_{\sigma}(r)dr\bigg).
\end{equation}
\end{thm}
\begin{proof}
 We integrate over $s$ and exponentiate equation (5.3). The assertion follows.
\end{proof}

We treat now case (b).
\begin{lem}
The logarithmic derivative $L_{S}(s)$ of the symmetrized zeta function 
satisfies the following functional equation
\begin{equation}
 L_{S}(s)+L_{S}(-s)=-8\pi \dim(V_{\chi}) \Vol(X) P_{\sigma}(s).
\end{equation}
\end{lem}
\begin{proof}
We consider (5.2) at $s\mapsto -s$. We subtract the resulting equation from (5.2).
Using the same argument as in Lemma 5.1, we have
\begin{equation*}
 \bigg(\prod_{j=2}^{N}\frac{1}{s_{j}^{2}-s^{2}}\bigg)\frac{1}{2s}(L_{S}(s)+L_{S}(-s))=-\bigg(\prod_{j=2}^{N}\frac{1}{s_{j}^{2}-s^{2}}\bigg)\frac{4\pi}{s}\dim(V_{\chi})\Vol(X)P_{\sigma}(s)
\end{equation*}
We multiply the above equation by the function 
\begin{equation*}
 2s\prod_{j=2}^{N}(s_{j}^{2}-s^{2}),
\end{equation*}
and get equation (5.5).
\end{proof}
\begin{thm}
 The symmetrized zeta function satisfies the following functional equation
\begin{equation}
 \frac{S(s;\sigma,\chi)}{S(-s,\sigma,\chi)}=\exp\bigg(-8\pi\dim(V_{\chi})\Vol(X)\int_{0}^{s}P_{\sigma}(r)dr\bigg).
\end{equation}
\end{thm}
\begin{proof}
We integrate over $s$ and exponentiate equation (5.5).
The assertion follows.
\end{proof}
\begin{thm}
 The super zeta function satisfies the functional equation
\begin{equation}
Z^{s}(s;\sigma,\chi)Z^{s}(-s;\sigma,\chi)=e^{2\pi i \eta(0,D^{\sharp}_{\chi}(\sigma))},
\end{equation}
where $\eta(0,D^{\sharp}_{\chi}(\sigma))$ is the function $\eta(s,D^{\sharp}_{\chi}(\sigma))$ as in (4.14) in Lemma 4.6, at $s=0$.
Furthermore,
\begin{equation}
 Z^{s}(0;\sigma,\chi)=e^{\pi i \eta(0,D^{\sharp}_{\chi}(\sigma))}.
\end{equation}
\end{thm}
\begin{proof}
By \cite[Proposition 7.7]{Spilioti2018}, we have
\begin{equation*}
 \Tr(D^{\sharp}_{\chi}(\sigma)\prod_{i=1}^{N}(D^{\sharp}_{\chi}(\sigma)^{2}+s_{i}^{2})^{-1})=\frac{-i}{2}
 \sum_{i=1}^{N}\bigg(\prod_{\substack{j=1\\ j\neq i}}^{N}\frac{1}{s_{j}^{2}-s_{i}^{2}}\bigg)L^{s}(s_{i}).
\end{equation*}
Equivalently, for $\text{Re}(s_{i}^2)\gg 0$, we get
\begin{align*}
 \sum_{i=1}^{N}\bigg(\prod_{\substack{j=1\\ j\neq i}}^{N}\frac{1}{s_{j}^{2}-s_{i}^{2}}\bigg)L^{s}(s_{i})&=2i
 \Tr(D^{\sharp}_{\chi}(\sigma)\prod_{i=1}^{N}(D^{\sharp}_{\chi}(\sigma)^{2}+s_{i}^{2})^{-1})\\
 &=2i\Tr\int_{0}^{\infty}\sum_{i=1}^{N}\bigg(\prod_{\substack{j=1\\ j\neq i}}^{N}\frac{1}{s_{j}^{2}-s_{i}^{2}}\bigg)e^{-ts_{i}^{2}}D^{\sharp}_{\chi}(\sigma)
 e^{-tD^{\sharp}_{\chi}(\sigma)^{2}}dt,
\end{align*}
Using the same argument as in the proof of Proposition 7.7 in \cite{Spilioti2018}, we obtain
\begin{align}
 \sum_{i=1}^{N}\bigg(\prod_{\substack{j=1\\ j\neq i}}^{N}\frac{1}{s_{j}^{2}-s_{i}^{2}}\bigg)L^{s}(s_{i})
 &=2i\int_{0}^{\infty}\sum_{i=1}^{N}\bigg(\prod_{\substack{j=1\\ j\neq i}}^{N}\frac{1}{s_{j}^{2}-s_{i}^{2}}\bigg)e^{-ts_{i}^{2}}
 \Tr(D^{\sharp}_{\chi}(\sigma)e^{-tD^{\sharp}_{\chi}(\sigma)^{2}})dt.
\end{align}
We fix now $s_{2},\ldots,s_{N}\in\C$ and let $s_{1}=s\in\C$ vary.
We choose $s_{i}$ such that Re$(s_{i}^{2})\rightarrow\infty$. 
In the right-hand side of (5.9), the integrals that include the exponentials $e^{-ts_{i}^{2}}$ are
\begin{equation*}
 \int_{0}^{\infty}e^{-ts_{i}^{2}}\Tr(D^{\sharp}_{\chi}(\sigma)e^{-tD^{\sharp}_{\chi}(\sigma)^{2}})dt.
\end{equation*}
 As $t\rightarrow \infty$, $\Tr(D^{\sharp}_{\chi}(\sigma)e^{-t(D^{\sharp}_{\chi}(\sigma))^{2}})$ and $e^{-ts_{i}^{2}}$ decay exponentially. As $t\rightarrow 0^{+}$, we use the asymptotic expansion of the trace of the operator 
$D^{\sharp}_{\chi}(\sigma)e^{-t(D^{\sharp}_{\chi}(\sigma))^{2}}$. By Lemma 4.1, we have
\begin{equation*}
 \Tr(D^{\sharp}_{\chi}(\sigma)e^{-tD^{\sharp}_{\chi}(\sigma)^{2}})\sim_{t\rightarrow 0}\dim(V_{\chi})(\alpha_{0}t^{1/2}+O(t^{3/2})),
\end{equation*}
where
\begin{equation*}
 \alpha_{0}=\int_{X}a_{0}(x)d\mu(x),
\end{equation*}
$a_{0}(x)$ is a smooth function  and $\mu(x)$ is the volume measure determined by the Riemannian metric on $X$.
Hence, each of the integrals 
\begin{equation*}
 \int_{0}^{\infty}e^{-ts_{i}^{2}}\Tr(D^{\sharp}_{\chi}(\sigma)e^{-tD^{\sharp}_{\chi}(\sigma)^{2}})dt,
\end{equation*}
is well defined. 
We multiply both sides of (5.9) by the finite product
\begin{equation*}
\prod_{j=2}^{N}s_{j}^{2}-s^{2}.
\end{equation*}
Then, we obtain
\begin{equation}
L^{s}(s)=2i\int_{0}^{\infty}e^{-ts^{2}}
 \Tr(D^{\sharp}_{\chi}(\sigma)e^{-tD^{\sharp}_{\chi}(\sigma)^{2}})dt
+\bigg(\prod_{j=2}^{N}s_{j}^{2}-s^{2}\bigg)\Xi(s_{2},\ldots,s_{N}).
\end{equation}
The term $\Xi(s_{2},\ldots,s_{N})$ comes form the terms that correspond to the summands over $i=2,\ldots,N$ and hence it has a fixed value in $\C$, since $s_{2},\ldots,s_{N}$ are fixed. 
The estimation of the sum on the right hand side of (3.14) on p. 163 of \cite{Spilioti2018} can be improved to show that the $\log$ of the Selberg zeta function is exponentially decreasing.
Hence, following a similar consideration, by (\ref{f:log der super}),
$L^{s}(s_{i})$ decreases exponentially, as $\text{Re}(s_{i})\rightarrow \infty$.
Therefore, the even polynomial 
arising from the term 
\begin{equation*}
\bigg(\prod_{j=2}^{N}s_{j}^{2}-s^{2}\bigg)\Xi(s_{2},\ldots,s_{N})
\end{equation*}
vanishes.
Hence, we have
\begin{equation*}
 L^{s}(s)=2i\int_{0}^{\infty}e^{-ts^{2}}\Tr(D^{\sharp}_{\chi}(\sigma)e^{-t(D^{\sharp}_{\chi}(\sigma))^{2}})dt.
\end{equation*}
Let $\Pi_{+}$ (resp. $\Pi_{-}$) be the projection on the span of the root spaces corresponding to eigenvalues $\lambda$ of $D^{\sharp}_{\chi}(\sigma)$
 with Re$(\lambda^{2})>0$ (resp. Re$(\lambda^{2})\leq 0$).
Recall that there are only finitely many eigenvalues of $D^{\sharp}_{\chi}(\sigma)$ such that Re$(\lambda^{2})\leq 0$.
We write
\begin{align}
 L^{s}(s)=\notag&2i\int_{0}^{\infty}e^{-ts^{2}}\Tr(\Pi_{+}D^{\sharp}_{\chi}(\sigma)e^{-tD^{\sharp}_{\chi}(\sigma)^{2}})dt\\
&+2i\int_{0}^{\infty}e^{-ts^{2}}\Tr(\Pi_{-}D^{\sharp}_{\chi}(\sigma)e^{-tD^{\sharp}_{\chi}(\sigma)^{2}})dt.
\end{align}
We set 
\begin{align*}
I_{+}(s):=\int_{0}^{\infty}e^{-ts^{2}}\Tr(\Pi_{+}D^{\sharp}_{\chi}(\sigma)e^{-tD^{\sharp}_{\chi}(\sigma)^{2}})dt\\
I_{-}(s):=\int_{0}^{\infty}e^{-ts^{2}}\Tr(\Pi_{-}D^{\sharp}_{\chi}(\sigma)e^{-tD^{\sharp}_{\chi}(\sigma)^{2}})dt.  
\end{align*}
For $s\in\C$ with Re$(s)>0$, we consider the integral 
\begin{align*}
 \int_{s}^{\infty}I_{+}(w)dw= \int_{s}^{\infty}\int_{0}^{\infty}e^{-tw^{2}}\Tr(\Pi_{+}D^{\sharp}_{\chi}(\sigma)e^{-tD^{\sharp}_{\chi}(\sigma)^{2}})dtdw\\
=\int_{0}^{\infty}\int_{s}^{\infty}e^{-tw^{2}}\Tr(\Pi_{+}D^{\sharp}_{\chi}(\sigma)e^{-tD^{\sharp}_{\chi}(\sigma)^{2}})dwdt.\\
\end{align*}
If we make the change of variables $w\mapsto \frac{1}{\sqrt{t}} u$, we get
\begin{align*}
 \int_{s}^{\infty}I_{+}(w)dw=\int_{0}^{\infty}\int_{\sqrt{t}s}^{\infty}\frac{1}{\sqrt{t}}e^{-u^{2}}\Tr(\Pi_{+}D^{\sharp}_{\chi}(\sigma)
 e^{-tD^{\sharp}_{\chi}(\sigma)^{2}})dudt.\\
\end{align*}
We use now the error function for  a complex number $z\in\C$,
\begin{equation*}
 \Phi(z):=\frac{2}{\sqrt{\pi}}\int_{0}^{z}e^{-u^2}du.
\end{equation*}
It holds
\begin{equation*}
 \frac{2}{\sqrt{\pi}}\int_{z}^{\infty}e^{-u^2}du=1-\frac{2}{\sqrt{\pi}}\int_{0}^{z}e^{-u^2}du=1-\Phi(z).
\end{equation*}
Hence,
\begin{equation}
\int_{s}^{\infty}I_{+}(w)dw=\int_{0}^{\infty}\frac{\sqrt{\pi}}{2\sqrt{t}}\bigg(1-\frac{2}{\sqrt{\pi}}\int_{0}^{\sqrt{t}s}e^{-u^{2}}du\bigg)
\Tr(\Pi_{+}D^{\sharp}_{\chi}(\sigma)e^{-tD^{\sharp}_{\chi}(\sigma)^{2}})dt,
\end{equation}
and
\begin{equation}
\int_{-s}^{\infty}I_{+}(w)dw=\int_{0}^{\infty}\frac{\sqrt{\pi}}{2\sqrt{t}}\bigg(1-\frac{2}{\sqrt{\pi}}\int_{0}^{-\sqrt{t}s}e^{-u^{2}}du\bigg)
\Tr(\Pi_{+}D^{\sharp}_{\chi}(\sigma)e^{-tD^{\sharp}_{\chi}(\sigma)^{2}})dt.
\end{equation}
We add together (5.11) and (5.12) to get
\begin{equation}
 \int_{s}^{\infty}I_{+}(w)dw+\int_{-s}^{\infty}I_{+}(w)dw=\int_{0}^{\infty}\frac{\sqrt{\pi}}{\sqrt{t}}
\Tr(\Pi_{+}D^{\sharp}_{\chi}(\sigma)e^{-tD^{\sharp}_{\chi}(\sigma)^{2}})dt.
\end{equation}
For $s\in\C$ with Re$(s)>0$, we consider now the integral $ \int_{s}^{\infty}I_{-}(w)dw$.
Since there are only finitely many eigenvalues of $D^{\sharp}_{\chi}(\sigma)$ with Re$(\lambda^{2})\leq 0$, we can interchange
the order of integration and write
\begin{align}
 \int_{s}^{\infty}I_{-}(w)dw\notag&= \int_{s}^{\infty}\int_{0}^{\infty}e^{-tw^{2}}\Tr(\Pi_{-}D^{\sharp}_{\chi}(\sigma)e^{-tD^{\sharp}_{\chi}(\sigma)^{2}})dtdw\\\notag
&=\sum_{\substack{\lambda\\ \text{Re}(\lambda^{2})\leq0}}\lambda\int_{s}^{\infty}\int_{0}^{\infty}e^{-tw^{2}}e^{-t\lambda^{2}}dtdw\\
&=\sum_{\substack{\lambda\\ \text{Re}(\lambda^{2})\leq0}}\lambda\int_{s}^{\infty}\frac{1}{w^{2}+\lambda^{2}}dw.
\end{align}
We substitute at (5.14) $s\mapsto -s$ and add the resulting equation to (5.14). Then, by change of variables,  $w\mapsto w'=w/\lambda$,  we obtain
\begin{align*}
 \int_{s}^{\infty}I_{-}(w)dw+&\int_{-s}^{\infty}I_{-}(w)dw
=\sum_{\substack{\text{Re}(\lambda)>0\\ \text{Re}(\lambda^{2})\leq0}}\pi-\sum_{\substack{\text{Re}(\lambda)<0\\ \text{Re}(\lambda^{2})\leq0}}\pi
\end{align*}
The sums over $\lambda$ in the equation above are finite, because we sum over $\lambda$ with Re$(\lambda^2)\leq 0$,
and there are only finitely many eigenvalues such that Re$(\lambda^2)\leq 0$. 
Recall the definition of $\eta_{0}(0,D^{\sharp}_{\chi}(\sigma))$ from Section 4:
\begin{equation*}
 \eta_{0}(s,D^{\sharp}_{\chi}(\sigma)):=\sum_{\substack{\text{Re}(\lambda)>0\\ \text{Re}(\lambda^{2})\leq0}}\lambda^{-s}-
\sum_{\substack{\text{Re}(\lambda)<0\\ \text{Re}(\lambda^{2})\leq 0}}\lambda^{-s}.
\end{equation*}
Then,
\begin{equation}
\int_{s}^{\infty}I_{-}(w)dw+\int_{-s}^{\infty}I_{-}(w)dw =\pi\eta_{0}(0,D^{\sharp}_{\chi}(\sigma)).\\
\end{equation}
By Lemma 4.6, we have,
 \begin{equation*}
 \eta(s,D^{\sharp}_{\chi}(\sigma))= \eta_{0}(s,D^{\sharp}_{\chi}(\sigma))+\frac{1}{\Gamma(\frac{s+1}{2})}\int_{0}^{\infty}\Tr(\Pi_{+}D^{\sharp}_{\chi}(\sigma) e^{-t(D^{\sharp}_{\chi}(\sigma))^{2}})t^{\frac{s-1}{2}}dt.
\end{equation*}
Hence, by (5.10), (5.13) and (5.15) we get
\begin{align}
 \log Z^{s}(s;\sigma,\chi)+\log Z^{s}(-s;\sigma,\chi)&=\int_{s}^{\infty}L^{s}(w)dw+\int_{-s}^{\infty}L^{s}(w)dw\\\notag
&=2i\int_{s}^{\infty}I_{+}(w)dw+2i\int_{-s}^{\infty}I_{+}(w)dw\\\notag
&+2i\int_{s}^{\infty}I_{-}(w)dw+2i\int_{-s}^{\infty}I_{-}(w)dw\\\notag
&=2\pi i\big(\eta_{1}(0,D^{\sharp}_{\chi}(\sigma))+\eta_{0}(0,D^{\sharp}_{\chi}(\sigma))\big)\\
&=2\pi i \eta (0,D^{\sharp}_{\chi}(\sigma)).
\end{align}
We mention here that the logarithm of the super zeta function 
\begin{equation}
 \log Z^{s}(s;\sigma,\chi)=\int_{s}^{\infty}L^{s}(w)dw,
\end{equation}
does depend on the choice of the path connecting $s$ and infinity.
However, in order to obtain the functional equations, we have to exponentiate
the integral in the right-hand side of (5.18). Since the residues of the singularities of the super zeta functions are all integers (\cite[Theorem 7.8]{Spilioti2018}), 
the exponential of this integral is independent of the choice of the path.
Equation (5.7) follows by exponentiation of (5.17).
Also, by (5.17),
\begin{align*}
 \log Z^{s}(s;\sigma,\chi)+\log Z^{s}(-s;\sigma,\chi)=2\pi i \eta (0,D^{\sharp}_{\chi}(\sigma))
\end{align*}
For $s=0$,
\begin{align*}
 2\log Z^{s}(0;\sigma,\chi)=2\pi i \eta (0,D^{\sharp}_{\chi}(\sigma))
\end{align*}
Equivalently,
\begin{align*}
 Z^{s}(0;\sigma,\chi)=e^{\pi i \eta (0,D^{\sharp}_{\chi}(\sigma))}.
\end{align*}
\end{proof}
We prove now the functional equation for the Selberg zeta function in case (b).
\begin{thm}
 The Selberg zeta function satisfies the following functional equation 
\begin{equation}
 \frac{Z(s;\sigma,\chi)}{Z(-s;w \sigma,\chi)}=e^{\pi i \eta(0,D^{\sharp}_{\chi}(\sigma))}\exp\bigg(-4\pi\dim(V_{\chi})\Vol(X)\int_{0}^{s}P_{\sigma}(r)dr\bigg).
\end{equation}
\end{thm}
\begin{proof}
By (3.3) and (3.4), we have
\begin{align*}
\frac{Z(s;\sigma,\chi)}{Z(-s;w \sigma,\chi)}&=\sqrt{\frac{S(s;\sigma,\chi)Z^{s}(s;\sigma,\chi)}{S(-s;w \sigma,\chi)Z^{s}(-s;w\sigma,\chi)}}\\
&=\sqrt{\frac{S(s;\sigma,\chi)Z^{s}(s;\sigma,\chi)Z^{s}(-s;\sigma,\chi)}{S(-s;\sigma,\chi)}}\\
&=e^{\pi i \eta(0,D^{\sharp}_{\chi}(\sigma))}\exp\bigg(-4\pi\dim(V_{\chi})\Vol(X)\int_{0}^{s}P_{\sigma}(r)dr\bigg),
\end{align*}
 where in the last equation we have employed Theorem 5.4 and 5.5.
 We mention here that we consider only the plus sign of the square root in the equation above since by the functional equations 
(5.8), the super zeta function at $s=0$ is equal to $e^{\pi i \eta (0,D^{\sharp}_{\chi}(\sigma))}$ only.
\end{proof}

\section{Functional equations for the Ruelle zeta function}

To prove the functional equations for the Ruelle zeta function, we have to consider Selberg zeta functions associated with representations of $M$, which are not irreducible.
In such case, for a finite dimensional representation $\sigma$ of $M$, we have
\begin{equation*}
Z(s; \sigma,\chi)=\prod_{\sigma'\in\widehat{M}}
Z(s;\sigma',\chi)^{[\sigma:\sigma']},
\end{equation*}
where $[\sigma:\sigma']$ is the multiplicity of
$\sigma'$ in $\sigma$.

Recall the Iwasawa decomposition $G=KAN$ of $G$. Let $\mathfrak{n}$ be the Lie algebra of $N$.
Let $\mathfrak{n}_{\C}$ be the complexification of $\mathfrak{n}$.
Let $\nu_{p}$ be the representation
of $MA$ in $\Lambda^{p}\mathfrak{n}_{\C}$,
given by the $p$-th exterior power of the adjoint representation
\begin{equation*}
 \nu_{p}:=\Lambda^{p}\Ad_{\mathfrak{n}_{\C}}\colon MA\rightarrow \GL(\Lambda^{p}\mathfrak{n}_{\C}),\quad p=0,1,\ldots,d-1.
\end{equation*}
We consider the following identification $\mathfrak{a}_{\C}^{*}\cong \C$. 
Let $\alpha>0$ be the unique positive
root of $(\mathfrak{g},\mathfrak{a})$. Let $\lambda\colon A\to \C^\times$ be the character,
defined by $\lambda(a)=e^{\alpha(\log a)}$. 
We consider also finite-dimensional, irreducible representations $(\psi_{p},V_{\psi_{p}})\in\widehat{M}$. 
We define the set $J_{p}$ as the subset $J_{p}\subseteq\{(\psi_{p},\lambda): \psi_{p}\in\widehat{M}, \lambda\in\C\}$
such that the following decomposition holds.
\begin{equation*}
\Lambda^{p}\mathfrak{n}_{\C}=\sum_{(\psi_{p},\lambda)\in J_{p}}V_{\psi_{p}}\otimes \C_{\lambda}.
\end{equation*}
By Poincar\'{e} duality (\cite[p. 122]{BO}), we have for $p<\frac{d-1}{2}$,
\begin{equation}
J_{d-1-p}\subset\{(\psi_{p},2\vert\rho\vert-\lambda)\colon\psi_{p}\in\widehat{M}, \lambda\in\C\}.
\end{equation}
Let $\mathfrak{p}_{d}=i \mathfrak{p}$.
Let $\mathfrak{g}_{d}$ be the subalgebra of $\mathfrak{g}_{\C}$ defined by 
$\mathfrak{g}_{d}=\mathfrak{k}\oplus \mathfrak{p}_{d}$. Let $G_{d}$ be the 
Lie group that corresponds to $\mathfrak{g}_{d}$. $G_{d}$ is the 
the compact real form of the complexification  $G_{\C}$ of $G$.
Let $\mathfrak{a}_{d}$  be the subalgebra of $\mathfrak{p}_{d}$ 
defined by $\mathfrak{a}_{d}=i\mathfrak{a}$. 
Let $A_{d}$ be the corresponding Lie group.
We set $L:=G_{d}/MA_{d}$,
which is a K\"{a}hler manifold (\cite[p. 123]{BO}).
For $\psi_{p}, \sigma\in\widehat{M}$, we define 
\begin{equation*}
Q_{\psi_{p}\otimes\sigma}(s):= \sum_{\sigma'\in\widehat{M}}\left[ (\psi_{p}\otimes\sigma)
\colon \sigma'\right] Q_{\sigma'}(s),
\end{equation*}
where $\left[ (\psi_{p}\otimes\sigma)\colon \sigma'\right]$ is  the multiplicity 
of $\sigma'$ in $\psi_{p}\otimes\sigma$.
\begin{lem}
Let $(\sigma,V_{\sigma})\in\widehat{M}$. Let $Q_{\psi_{p}\otimes\sigma}(s),s\in \C$, be the Plancherel measure associated with the representation $\psi_{p}\otimes\sigma$, $p=0,\ldots,d-1$.
Let $f(s)$ be the polynomial of $s$ given by
\begin{align}
 f(s)&:=\notag(-1)^{\frac{d-1}{2}}Q_{\psi_{\frac{d-1}{2}}\otimes\sigma}(s)\\
&+\notag\sum_{p=0}^{\frac{d-3}{2}} \sum_{(\psi_{p},\lambda)\in J_{p}}(-1)^{p} [Q_{\psi_{p}\otimes\sigma}(s+\vert\rho\vert-\lambda)+Q_{\psi_{p}\otimes\sigma}(s-\vert\rho\vert+\lambda)]\\
&=\sum_{p=0}^{d-1}\sum_{(\psi_{p},\lambda)\in J_{p}}(-1)^{p}Q_{\psi_{p}\otimes \sigma}(s;\rho,\lambda).                     
\end{align}
Then,
\begin{equation}
f(s)=(d+1)\dim(V_{\sigma}).
\end{equation}
\end{lem}
\begin{proof}
As in the proof of Theorem 4.4 in \cite[p. 127]{BO}, the polynomial $f(s)$ is given by
\begin{equation*}
f(s)=\dim(V_{\sigma})\chi(L).
\end{equation*}
The Euler characteristic $\chi(L)$ is also calculated in \cite[p.127]{BO}:
\begin{equation*}
\chi(L)= d+1.
\end{equation*}
The assertion follows.
\end{proof}

\begin{thm}
 The Ruelle zeta function satisfies the following functional equation
\begin{equation}
 \frac{R(s;\sigma,\chi)}{R(-s;\sigma,\chi)}=\exp\bigg((-1)^{\frac{d-1}{2}+1}2\pi(d+1)\dim(V_{\sigma})\dim (V_{\chi})\frac{\Vol(X)}{\Vol(S^{d})}s\bigg). 
\end{equation}
\end{thm}
\begin{proof}
By \cite[Theorem 7.12]{Spilioti2018}, we have the following representation of the Ruelle zeta function.
\begin{equation}
R(s;\sigma,\chi)=\prod_{p=0}^{d-1}\bigg(\prod_{(\psi_{p},\lambda)\in J_{p}}Z(s+\vert\rho\vert-\lambda;\psi_{p}\otimes\sigma,\chi)\bigg)^{(-1)^{p}}.
\end{equation}
Using Poncar{\'e} duality, i.e., (6.1), and considering 
\begin{equation*}
s+\vert\rho\vert-\lambda\mapsto s+\vert\rho\vert-(2\vert\rho\vert-\lambda)=s-\vert\rho\vert+\lambda,
\end{equation*}
we get
\begin{align*}
R(s;\sigma,\chi)=\notag&Z(s;\psi_{\frac{d-1}{2}}\otimes\sigma,\chi)^{(-1)^{\frac{d-1}{2}}}\\
&\prod_{p=0}^{\frac{d-3}{2}} \prod_{(\psi_{p},\lambda)\in J_{p}}(Z(s+\vert\rho\vert-\lambda;\psi_{p}\otimes\sigma,\chi)Z(s-\vert\rho\vert+\lambda;\psi_{p}\otimes\sigma,\chi))^{(-1)^{p}}.
\end{align*}
Hence,
\begin{align*}
\frac{R(s;\sigma,\chi)}{R(-s;\sigma,\chi)}=\notag&\bigg(\frac{Z(s;\psi_{\frac{d-1}{2}}\otimes\sigma,\chi)}{Z(-s;\psi_{\frac{d-1}{2}}\otimes\sigma,\chi)}\bigg) ^{(-1)^{\frac{d-1}{2}}}\\
&\prod_{p=0}^{\frac{d-3}{2}} \prod_{(\psi_{p},\lambda)\in J_{p}}\bigg(\frac{Z(s+\vert\rho\vert-\lambda;\psi_{p}\otimes\sigma,\chi)
Z(s-\vert\rho\vert+\lambda;\psi_{p}\otimes\sigma,\chi)}{Z(-s+\vert\rho\vert-\lambda;\psi_{p}\otimes\sigma,\chi)Z(-s-\vert\rho\vert+\lambda;\psi_{p}\otimes\sigma,\chi)}\bigg)^{(-1)^{p}}.
\end{align*}
By Theorem 5.2, we get
\begin{align}
\frac{R(s;\sigma,\chi)}{R(-s;\sigma,\chi)}=\notag&\exp\Bigg\{-4\pi \dim (V_{\chi})\Vol(X)\bigg((-1)^{\frac{d-1}{2}}\int_{0}^{s}P_{\psi_{\frac{d-1}{2}}\otimes\sigma}(r)dr+\\
&\sum_{p=0}^{\frac{d-3}{2}} \sum_{(\psi_{p},\lambda)\in J_{p}}(-1)^{p} \bigg(\int_{0}^{s+\vert\rho\vert-\lambda}P_{\psi_{p}\otimes\sigma}(r)dr 
+\int_{0}^{s-\vert\rho\vert+\lambda}P_{\psi_{p}\otimes\sigma}(r)dr
\bigg)\bigg)\Bigg\}.
\end{align}
We set 
\begin{align*}
 F(s)=(-1)^{\frac{d-1}{2}}\int_{0}^{s}Q_{\psi_{\frac{d-1}{2}}\otimes\sigma}(r)dr+\sum_{p=0}^{\frac{d-3}{2}} &\sum_{(\psi_{p},\lambda)\in J_{p}}(-1)^{p}\bigg(\int_{0}^{s+\vert\rho\vert-\lambda}Q_{\psi_{p}\otimes\sigma}(r)dr\\
&+\int_{0}^{s-\vert\rho\vert+\lambda}Q_{\psi_{p}\otimes\sigma}(r)dr\bigg).
\end{align*}
Then, 
\begin{equation*}
\frac{d}{ds}F(s)=f(s),
\end{equation*}
where $f(s)$ is as in (6.2). Using Lemma 6.1, we get
\begin{equation}
\frac{R(s;\sigma,\chi)}{R(-s;\sigma,\chi)}=\exp\bigg(4\pi c(n)\dim (V_{\chi})\Vol(X)[(d+1)\dim(V_{\sigma})s+C]\bigg),
\end{equation}
where 
$c(n)$ is as in (2.9) and $C\in\R$ is a real constant.
On the other hand, if we set $s=0$ in (6.7), we get $1=\exp(4\pi c(n)\dim (V_{\chi})\Vol(X) C)$, and hence $C=0$. 
The assertion follows.
\end{proof}
We examine now case (b). Let $\tau_{p}$ be the standard representation
of $K$ in $\Lambda^{p}\R^{d}\otimes \C$.
Let $(\sigma_{p},V_{\sigma_{p}})$ be the standard representation of $M$ in $\Lambda^{p}\R^{d-1}\otimes \C$.
Let $\alpha>0$ be the unique positive
root of $(\mathfrak{g},\mathfrak{a})$. Let $\lambda\equiv\lambda_p\colon A\to \C^\times$ be the character,
defined by $\lambda\equiv\lambda_p(a)=e^{p\alpha(\log a)}$. Then, as a representation of 
$MA$, one has $\nu_p=\sigma_p\otimes\lambda_{p}$. 
We denote by $\C_{p}\cong \C$ the representation space of $\lambda_{p}$. 
Then, in the sense of $MA$-modules, we have
\begin{equation}
\Lambda^{p}\mathfrak{n}_{\C}=\Lambda^{p}\R^{d-1}\otimes \C_{p}.
\end{equation}
Let $D^{\sharp}_{\chi}(\sigma)$ be the twisted Dirac operator acting on $C^{\infty}(X,E_{\tau_{s}(\sigma)}\otimes E_{\chi})$.
We introduce here the twist $\widetilde{D^{\sharp}_{\chi}(\sigma)}$ of the Dirac operator $D^{\sharp}_{\chi}(\sigma)$ acting on
smooth sections of the vector bundle
\begin{equation*}
V_{\tau_{s}(\sigma)}\otimes V_{\chi}\otimes\bigg(\sum_{k=0}^{d-1}(-1)^{k}(d-k)\Lambda^{k}T^{*}X\bigg).
\end{equation*}
The twisted Dirac operator $\widetilde{D^{\sharp}_{\chi}(\sigma)}$ is defined in a similar way as the Dirac operator $D^{\sharp}_{\chi}(\sigma)$ in Section 4,
equipping the bundle $\Lambda^{k}T^{*}X$ with the Levi-Civita connection of $X$.
\begin{thm}
The super Ruelle zeta function associated with a non-Weyl invariant representation $\sigma\in\widehat{M}$ 
satisfies the functional equation
\begin{equation}
 R^{s}(s;\sigma,\chi) R^{s}(-s;\sigma,\chi)=e^{2i\pi\eta(\widetilde{D^{\sharp}_{\chi}(\sigma)})},
\end{equation}
where $\eta(\widetilde{D^{\sharp}_{\chi}(\sigma)})$ denotes the eta invariant of the twisted
Dirac operator $\widetilde{D^{\sharp}_{\chi}(\sigma)}$.\\
Moreover, the following equation holds
\begin{equation}
\frac{R(s;\sigma,\chi)}{ R(-s;w\sigma,\chi)}= e^{i\pi\eta(\widetilde{D^{\sharp}_{\chi}(\sigma)})}\exp\bigg((-1)^{\frac{d-1}{2}+1}2\pi(d+1)\dim(V_{\sigma})\dim (V_{\chi})\frac{\Vol(X)}{\Vol(S^{d})}s\bigg).
\end{equation}
\end{thm}
\begin{proof}
By \cite[p. 23]{BO}, we have 
\begin{align*}
 &\sigma_{p}=i^{*}((-1)^{0}\tau_{p}+(-1)^{1}\tau_{p-1}+\ldots+(-1)^{p-1}(\tau_{1}-\Id)),\quad p=1,2,\ldots d-1\\
&s^{+}+s^{-}=i^{*}(s),\quad\text{otherwise}.
\end{align*}
If we take the alternating sum of $\sigma_{p}$ over $p$ we get
\begin{equation}
\sum_{p=0}^{d-1}(-1)^{p}\sigma_{p}=i^{*}\big(\sum_{k=0}^{d-1}(-1)^{k}(d-k)\tau_{k}\big).
\end{equation}
By the definition of the super Ruelle zeta function, we have
\begin{align}
R^{s}(s;\sigma,\chi) R^{s}(-s;\sigma,\chi)
=\frac{R(s;\sigma,\chi)}{R(-s;w \sigma,\chi)}\frac{R(-s;\sigma,\chi)}{R(s;w\sigma,\chi)}.
\end{align}
We use now the representation (6.5) of the Ruelle zeta function. By the Poincar\'{e} duality we obtain
\begin{align*}
  R(s;\sigma,\chi)=\notag&Z(s;\sigma_{\frac{d-1}{2}}\otimes\sigma,\chi)^{(-1)^{\frac{d-1}{2}}}\\
&\prod_{p=0}^{\frac{d-3}{2}}(Z(s+\vert\rho\vert-\lambda;\sigma_{p}\otimes\sigma,\chi)Z(s-\vert\rho\vert+\lambda;\sigma_{p}\otimes\sigma,\chi))^{(-1)^{p}}.
\end{align*}
If we substitute this expression in the right-hand side of (6.12), we have
\begin{align*}
R^{s}(s;\sigma,\chi)& R^{s}(-s;\sigma,\chi)\notag={\frac{Z(s;\sigma_{\frac{d-1}{2}}\otimes\sigma,\chi)}{Z(-s;\sigma_{\frac{d-1}{2}}\otimes w \sigma,\chi)}}^{(-1)^{\frac{d-1}{2}}}\\
& \cdot\frac{\prod_{p=0}^{\frac{d-3}{2}}(Z(s+\vert\rho\vert-\lambda;\sigma_{p}\otimes\sigma,\chi)Z(s-\vert\rho\vert+\lambda;\sigma_{p}\otimes\sigma,\chi))^{(-1)^{p}}}
{\prod_{p=0}^{\frac{d-3}{2}}(Z(-s+\vert\rho\vert-\lambda;\sigma_{p}\otimes w\sigma,\chi)Z(-s-\vert\rho\vert+\lambda;\sigma_{p}\otimes w\sigma,\chi))^{(-1)^{p}}}\\
&\cdot {\frac{Z(-s;\sigma_{\frac{d-1}{2}}\otimes\sigma,\chi)}{Z(s;\sigma_{\frac{d-1}{2}}\otimes w \sigma,\chi)}}^{(-1)^{\frac{d-1}{2}}}\\
&\cdot\frac{\prod_{p=0}^{\frac{d-3}{2}}(Z(-s+\vert\rho\vert-\lambda;\sigma_{p}\otimes\sigma,\chi)Z(-s-\vert\rho\vert+\lambda;\sigma_{p}\otimes\sigma,\chi))^{(-1)^{p}}}
{\prod_{p=0}^{\frac{d-3}{2}}(Z(s+\vert\rho\vert-\lambda;\sigma_{p}\otimes w\sigma,\chi)Z(s-\vert\rho\vert+\lambda;\sigma_{p}\otimes w\sigma,\chi))^{(-1)^{p}}}.\\
\end{align*}
By Theorem 5.6, we get
\begin{align*}
R^{s}(s;\sigma,\chi)& R^{s}(-s;\sigma,\chi)\notag=(e^{2i\pi\eta(0,D^{\sharp}_{\chi}(\sigma\otimes\sigma_{d-1/2})})^{(-1)^{\frac{d-1}{2}}}\prod_{p=0}^{\frac{d-3}{2}}(e^{2i\pi\eta(0,D^{\sharp}_{\chi}
(\sigma\otimes\sigma_{p}))})^{(-1)^{p}}.
\end{align*}
where we used the fact that the Plancherel polynomial is an even function.
Finally, we have
\begin{align}
R^{s}(s;\sigma,\chi)R^{s}(-s;\sigma,\chi)\notag&=e^{2i\pi\sum_{p=0}^{d-1}(-1)^{p}\eta(0,D^{\sharp}_{\chi}(\sigma\otimes\sigma_{p}))}\\
&=e^{2i\pi\eta(\widetilde{D^{\sharp}_{\chi}(\sigma)})},
\end{align}
where $\eta(\widetilde{D^{\sharp}_{\chi}(\sigma)})=\sum_{p=0}^{d-1}(-1)^{p}\eta(0,D^{\sharp}_{\chi}(\sigma\otimes\sigma_{p}))$.
For the functional equations (6.10), we have
\begin{align*}
\frac{R(s;\sigma,\chi)^{2}}{R(-s;w\sigma,\chi)^{2}}&=\frac{R(s;\sigma,\chi)}{R(-s;w \sigma,\chi)}\frac{R(-s;\sigma,\chi)}{R(s;w \sigma,\chi)}
\frac{R(s;\sigma,\chi)}{R(-s;\sigma,\chi)}\frac{R(s;w \sigma,\chi)}{R(-s;w \sigma,\chi)}\\
&=e^{2i\pi\eta(\widetilde{D^{\sharp}_{\chi}(\sigma)})}\frac{R(s;\sigma,\chi)R(s;w \sigma,\chi)}{R(-s;\sigma,\chi)R(-s;w \sigma,\chi)}.\\
\end{align*}
By (6.4), we get
\begin{equation*}
\frac{R(s;w \sigma,\chi)}{R(-s;w \sigma,\chi)}=
\exp\bigg((-1)^{\frac{d-1}{2}+1}2\pi(d+1)\dim(V_{\sigma})\dim (V_{\chi})\frac{\Vol(X)}{\Vol(S^{d})}s\bigg).
\end{equation*}
Hence, 
\begin{equation}
\frac{R(s;\sigma,\chi)^{2}}{R(-s;w\sigma,\chi)^{2}}=e^{2i\pi\eta(\widetilde{D^{\sharp}_{\chi}(\sigma)})}\exp2\bigg((-1)^{\frac{d-1}{2}+1}2\pi(d+1)\dim(V_{\sigma})\dim (V_{\chi})\frac{\Vol(X)}{\Vol(S^{d})}s\bigg).
\end{equation}
By (6.13), we have
\begin{align*}
\log (R^s(s;\sigma,\chi))+\log (R^s(-s;\sigma,\chi))=2i\pi\eta(\widetilde{D^{\sharp}_{\chi}(\sigma)})
\end{align*}
For $s=0$,
\begin{align*}
2\log (R^s(0;\sigma,\chi))=2i\pi\eta(\widetilde{D^{\sharp}_{\chi}(\sigma)})
\end{align*}
Equivalently,
\begin{align}
R^s(0;\sigma,\chi)=\exp(\pi i \eta(\widetilde{D^{\sharp}_{\chi}(\sigma)}).
\end{align}
Hence, by (6.14)
\begin{equation*}
\frac{R(s;\sigma,\chi)}{R(-s;w\sigma,\chi)}=e^{\pi i \eta(\widetilde{D^{\sharp}_{\chi}(\sigma)})}\exp\bigg((-1)^{\frac{d-1}{2}+1}2\pi(d+1)\dim(V_{\sigma})\dim (V_{\chi})\frac{\Vol(X)}{\Vol(S^{d})}s\bigg).
\end{equation*}
\end{proof}

\section{The determinant formula}

We recall here the notion of the graded regularized determinant of an elliptic differential operator.
Let $E=E_{+}\oplus E_{-}$ be a $\Z_{2}$-graded  vector bundle over a compact
Riemannian manifold $X$. Let $P: C^{\infty}(X,E)\rightarrow C^{\infty}(X,E)$ be an elliptic differential operator,
which is bounded from below. We assume that $P$ preserves the grading, i.e.,
we assume that with respect to the decomposition
\begin{equation*}
C^{\infty}(X,E)=C^{\infty}(X,E^{+})\oplus C^{\infty}(X,E^{-}),
\end{equation*}
$P$ takes the form,
\begin{equation*}
   P=
  \left( {\begin{array}{cc}
   P{+} & 0 \\
   0 &    P_{-} \\
  \end{array} } \right).
\end{equation*}
Then, the graded determinant $\det_{\gr}(P)$ of $P$ is defined by
\begin{equation*}
{\det}_{\gr}(P)=\frac{\det(P_{+})}{\det(P_{-})},
\end{equation*}
where $\det(P_{+})$ and $ \det(P_{-})$ denote the 
regularized determinants of the operators $P_{+}$ and 
$P_{-}$, respectively (see  \cite[Definition 3.6]{BK2}).

As it is mentioned before, $E(\sigma)$
is a $\Z_{2}$-graded locally homogeneous vector bundle 
over $X$ (\cite[p. 27, 29]{BO}, \cite[p. 175]{Spilioti2018}).
Moreover, $A_{\chi}^{\sharp}(\sigma)$ acting on $C^{\infty}(X, E(\sigma)\otimes E_{\chi})$
preserves the grading (\cite[p. 175]{Spilioti2018}). 
Hence, we consider the super trace $\Tr_{s}(e^{-tA_{\chi}^{\sharp}(\sigma)})$
in Definition 7.1 of the xi function $\xi(z,s;\sigma)$ and  the generalized zeta function 
$\zeta(z,s;\sigma)$ below.
In addition, the regularized determinants of the  operators $A_{\chi}^{\sharp}(\sigma)+s^{2}$
in Theorem 7.8 and Proposition 7.9 below are graded regularized determinants, i.e.,
we consider
\begin{equation*}
{\det}_{\gr}(A_{\chi}^{\sharp}(\sigma)+s^{2})=\frac{\det(A_{\chi,+}^{\sharp}(\sigma)+s^{2})}{\det(A_{\chi,-}^{\sharp}(\sigma)+s^{2})}.
\end{equation*}
\textbf{Remark:}
If $\theta$ is an Agmon angle for $A_{\chi,+}^{\sharp}(\sigma)+s^{2}$
and $A_{\chi,-}^{\sharp}(\sigma)+s^{2}$, 
then the corresponding  regularized 
determinants do not depend on the choice of the Agmon angle $\theta$, since the operators
have a self-adjoint principal symbol (\cite[Section 3.10]{BK2}). 
Hence, there is no notion of $\theta$ in our definition.
\newline

By Lemma 4.2, there exist coefficients $a_{j}\in\C$, $j=0, 1, \ldots $ such that
\begin{equation}
\Tr(e^{-tA_{\chi,\pm}^{\sharp}(\sigma)})\sim_{t\rightarrow 0^{+}}\dim(V_{\chi})\sum_{j=0}^{\infty}a_{j}^{\pm}t^{\frac{j-d}{2}}.
\end{equation}
\begin{defi}
Let $C\in\R$ be such that $\text{Re}(\lambda)>C$ for all 
$\lambda\in \spec(A_{\chi}^{\sharp}(\sigma))$. For $\text{\Re}(z)>d/2$ and Re$(s)>-C$, we define the xi function associated to the operator $A_{\chi}^{\sharp}(\sigma)$ by
\begin{equation}
\xi(z,s;\sigma):=\int_{0}^{\infty}e^{-ts}\Tr_{s}({e^{-tA_{\chi}^{\sharp}(\sigma)}})t^{z-1}dt,
\end{equation}
and the generalized zeta function  by
\begin{equation}
\zeta(z,s;\sigma):=\frac{1}{\Gamma(z)}\int_{0}^{\infty}e^{-ts}\Tr_{s}(e^{-tA_{\chi}^{\sharp}(\sigma)})t^{z-1}dt.
\end{equation} 
\end{defi}
We define the theta function $\theta(t)$ associated with the operator $e^{-tA_{\chi}^{\sharp}(\sigma)}$ by
\begin{equation*}
 \theta(t):=\Tr_{s}({e^{-tA_{\chi}^{\sharp}(\sigma)}})
\end{equation*}
We set
\begin{align*}
 \theta_{1}(t):=\Tr({e^{-tA_{\chi,+}^{\sharp}(\sigma)}});\\
  \theta_{2}(t):=\Tr({e^{-tA_{\chi,-}^{\sharp}(\sigma)}}).
\end{align*}
Then,
\begin{equation*}
 \theta(t)= \theta_{1}(t)- \theta_{2}(t).
\end{equation*}
We set
\begin{align*}
\xi_{1}(z,s;\sigma):=\int_{0}^{\infty}e^{-ts}\Tr({e^{-tA_{\chi,+}^{\sharp}(\sigma)}})t^{z-1}dt;\\
\xi_{2}(z,s;\sigma):=\int_{0}^{\infty}e^{-ts}\Tr({e^{-tA_{\chi,-}^{\sharp}(\sigma)}})t^{z-1}dt.
\end{align*}
Then, $\xi(z,s;\sigma)$ is just the Mellin transform of $e^{-ts}\theta(t)$
and moreover
\begin{equation*}
\xi(z,s;\sigma)=\xi_{1}(z,s;\sigma)-\xi_{2}(z,s;\sigma).
\end{equation*}
We study here $\xi_{1}(z,s;\sigma)$. One can proceed similarly for $\xi_{2}(z,s;\sigma)$.
For $\lambda_{j}\in \spec(A_{\chi,+}^{\sharp}(\sigma))$, let $L_{j}^{+}$ be the horizontal half line going from $-\infty$ to $-\lambda_{j}$. We define $U_{L}^{+}$ to be
the complement of all the half lines $L_{j}^{+}$ in $\C$.
We denote by $\N_{0}$ the set $\N_{0}=\N\cup \{0\} $.
\begin{lem}
For  $s\in U_{L}^{+}$, $\xi_{1}(z,s;\sigma)$ admits a meromorphic continuation to $z\in\C$. Furthermore, it has simple poles at $k_{r}=-\big(\frac{r-d}{2}\big)$, where $r\in\N_{0}$.
\end{lem}
\begin{proof}
We denote by $m(\lambda)$ the algebraic multiplicity of the eigenvalue $\lambda\in\spec(A_{\chi,+}^{\sharp}(\sigma))$. We consider the following ordering $\text{Re}(\lambda_{1})\leq\text{Re}(\lambda_{2})\leq\text{Re}(\lambda_{3})\leq\ldots$ of the real parts of the eigenvalues of $A_{\chi,+}^{\sharp}(\sigma)$.
We observe that for every positive real number $c$,
there are only finitely many eigenvalues $\lambda_{j}$ such that $\text{Re}(\lambda_{j})\leq c$.
Hence, there exists a positive integer $N\geq 1$, such that $\text{Re}(\lambda_{j})>c$, for every $j>N$. Then, we have
\begin{align}
\bigg|\sum_{j=1}^{\infty}m(\lambda_{j})e^{-t\lambda_{j}}-&\notag\sum_{j=1}^{N}m(\lambda_{j})e^{-t\lambda_{j}}\bigg|
=\bigg |\sum_{j=N+1}^{\infty}m(\lambda_{j})e^{-t\lambda_{j}}\bigg |\\
&\leq \sum_{j=N+1}^{\infty}m(\lambda_{j})e^{-t\text{Re}(\lambda_{j})}.
\end{align}
Then, we have for $t\geq 1$
\begin{align}
 \sum_{j=N+1}^{\infty}m(\lambda_{j})e^{-t\text{Re}(\lambda_{j})}
\leq e^{-tc/2}\sum_{j=N+1}^{\infty}m(\lambda_{j})e^{-\text{Re}(\lambda_{j})/2}.
\end{align}
To estimate the last sum, we will use Weyl's law for the non-self-adjoint operator $A_{\chi,+}^{\sharp}(\sigma)$. 
Given a positive constant $k$, we define the counting function $\mathcal{N}(k)$ by
\begin{equation*}
 \mathcal{N}(k):=\sum_{\substack{\lambda\in\spec(A_{\chi,+}^{\sharp}(\sigma))\\{|\lambda|\leq k}}} m(\lambda).
\end{equation*}
In \cite{M1}, the generalization of  Weyl's law for the non-self-adjoint case is proved.
By \cite[Lemma 2.2]{M1}, we have
\begin{equation}
 \mathcal{N}(k)=\frac{\rank(E(\sigma)\otimes E_{\chi})\Vol(X)}{(4\pi)^{d/2}\Gamma(d/2+1)}k^{d/2}+o(k^{d/2}), \quad k\rightarrow \infty,
\end{equation}
where $\rank(E(\sigma)\otimes E_{\chi})$ denotes the rank of the product vector bundle $E(\sigma)\otimes E_{\chi}$.
Let  $a>0$ be the slope of the boundary of the cone, in which all the eigenvalues $\lambda_{j}$ of $A^{\sharp}_{\chi,+}(\sigma)$ are contained. We have
\begin{equation*}
 \sharp\{j\colon|\text{Re}(\lambda_{j})|\leq \lambda\}\leq \sharp\{j\colon|\lambda_{j}|\leq \sqrt{1+a^{2}}
 \lambda\}\leq \mathcal{N}(\sqrt{1+a^{2}}\lambda).
\end{equation*}
By (7.6), similarly to (4.21), we get
\begin{align}
\sum_{j=N+1}^{\infty}m(\lambda_{j})e^{-\text{Re}(\lambda_{j})/2}&
<\infty.
\end{align}
where $C_{1}$ is a positive constant.

Hence, by (7.4), (7.5), (7.7) and the definition of $\theta_{1}(t)$, we have that given a positive number $c>0$, there exist a positive integer $N$ and a $K>0$ such that
\begin{subequations}
\begin{equation}
 \bigg|\theta_{1}(t)-\sum_{j=1}^{N}m(\lambda_{j})e^{-t\lambda_{j}}\bigg|\leq Ke^{-ct},\quad t\geq 1.
\end{equation}
Furthermore, by the asymptotic expansion of the trace of the operator $e^{-tA_{\chi,+}^{\sharp}(\sigma)}$ (7.1), we have that for every positive integer $N$,
\begin{equation}
 \theta_{1}(t)-\sum_{j=0}^{N}a_{j}t^{\frac{j-d}{2}}=O(t^{\frac{N+1-d}{2}}), \quad t\rightarrow 0.
\end{equation}
\end{subequations}
By (7.8a) and (7.8b), $\theta_{1}(t)$ satisfies the assumptions as in \cite[$\AS$ 1, $\AS$ 2, p. 16]{JL}.
Hence, we can apply \cite[Theorem 1.5]{JL} for $p=\frac{j-d}{2}$
and obtain the meromorphic continuation of $\xi_{1}(z,s;\sigma)$ to $z\in\C$.
The simple poles are located at $k_{r}=-\big(\frac{r-d}{2}\big)$, where $r\in\N_{0}$.
\end{proof}

Let now $\lambda_{j}\in \spec(A_{\chi,-}^{\sharp}(\sigma))$ and $L_{j}^{-}$ be the horizontal half line going from $-\infty$ to $-\lambda_{j}$. We define $U_{L}^{-}$ to be
the complement of all the half lines $L_{j}^{-}$ in $\C$. Then, 
similarly, for $\xi_{2}(z,s;\sigma)$, we have  the following lemma.
\begin{lem}
For  $s\in U_{L}^{+}$,  $\xi_{2}(z,s;\sigma)$ admits a meromorphic continuation to $z\in\C$. Furthermore, it has simple poles at $k_{j,n}=-(\frac{j-d}{2}+n)$, where $n\in \N_{0} $.
\end{lem}

\begin{thm}
For  $s\in U_{L}^{+}$,  $\xi_{1}(z,s;\sigma)$
is holomorphic at $z=0$.
\end{thm}
\begin{proof}
By \cite[Theorem 1.6]{JL}, $\xi_{1}(z,s;\sigma)$ has Laurent expansion at $z=0$
\begin{equation*}
\xi_{1}(z,s;\sigma)=R_{-1}(0,s;\sigma)z^{-1}+R_{0}(0,s;\sigma)+R_{1}(0,s;\sigma)z+\ldots,
\end{equation*}
where $R_{-1}(0,s;\sigma)$ is a polynomial of degree $\leq d/2$.
In our case the pole at 
$z=0$ is simple since the number $n(0')$ is zero. For justification, see the definition
of $n(0')$ in \cite[p. 17]{JL} and notice that the asymptotic expansion (7.1) corresponds to the so called special case, i.e., the coefficients in (7.1) are constants and there are no logarithmic terms (see \cite[p. 15--16]{JL}).
By the proof of Theorem 1.6 in \cite{JL}, the polynomial $R_{-1}(0,s;\sigma)$,
can be explicitly given
\begin{equation}
R_{-1}(0,s;\sigma)=\sum_{j/2-d/2+k=0}(-1)^{k}\frac{s^{k}}{k!}a_{j}^{+}, 
\end{equation}
where $a_{j}^{+}$ are the coefficients in the short time asymptotic expansion (7.1) of the heat operator, with $j/2-d/2<0$.
By \cite[Lemma 1.7.4 (d)]{Gilk}, the coefficients $a_{j}^{+}$ vanish for $j$ odd.
Hence, by (7.9) and the fact that $d=\dim X$ is odd, 
$R_{-1}(0,s;\sigma)=0$. The assertion follows.
\end{proof}
Similarly, for $\xi_{2}(z,s;\sigma)$, we have the following theorem.
\begin{thm}
For  $s\in U_{L}^{-}$,  $\xi_{2}(z,s;\sigma)$ is holomorphic at $z=0$.
\end{thm}
We set $U_{L}:=U_{L,+}\cup U_{L,-}$.
\begin{coro}
For $s\in U_{L}$, the generalized zeta function  $\zeta(z,s;\sigma)$ is holomorphic at $z=0$.
In addition, 
\begin{equation}
\frac{d}{dz} \zeta(z,s;\sigma)\bigg|_{z=0}= \xi(0,s;\sigma).
\end{equation}
\end{coro}
\begin{proof}
The generalized zeta function is by definition the xi function divided by $\Gamma(z)$.
We use
\begin{equation*}
\frac{1}{\Gamma(z)}=z+\gamma z^{2}+O(z^3),
\end{equation*}
where $\gamma$ is the Euler constant.
We have
\begin{align*}
\frac{d}{dz} \zeta(z,s;\sigma)\bigg|_{z=0}&=\frac{d}{dz}\bigg(\frac{1}{\Gamma(z)}\xi_{1}(z,s;\sigma)- \frac{1}{\Gamma(z)}\xi_{1}(z,s;\sigma)\bigg)\bigg|_{z=0}\\
&=\xi_{1}(0,s;\sigma)-\xi_{2}(0,s;\sigma)\\
&=\xi(0,s;\sigma)
\end{align*}
\end{proof}

\begin{defi}
The graded regularized determinant of the operator $A_{\chi}^{\sharp}(\sigma)+s^{2}$ is defined by
\begin{equation}
{\det}_{\gr}(A_{\chi}^{\sharp}(\sigma)+s^{2}):=\exp\bigg(-\frac{d}{dz}\zeta(z,s^{2};\sigma)\bigg|_{z=0}\bigg).
\end{equation}
\end{defi}
By (7.10), we get
\begin{equation*}
{\det}_{\gr}(A_{\chi}^{\sharp}(\sigma)+s^{2})=\exp(-\xi(0,s^{2};\sigma)).
\end{equation*}
Equivalently,
\begin{equation}
\log({\det}_{\gr}(A_{\chi}^{\sharp}(\sigma)+s^{2}))=-\xi(0,s^{2};\sigma)
=-(\xi_{1}(0,s^{2};\sigma)-\xi_{2}(0,s^{2};\sigma)).
\end{equation}

\begin{thm}
Let  ${\det}_{\gr}(A_{\chi}^{\sharp}(\sigma)+s^{2})$ be the regularized determinant associated with the operator $A_{\chi}^{\sharp}(\sigma)+s^{2}$.
Then,
\begin{enumerate}
\item  $\bf {case (a)}$ the Selberg zeta function has the representation
\begin{equation}
Z(s;\sigma,\chi)={\det}_{\gr}(A_{\chi}^{\sharp}(\sigma)+s^{2})\exp\bigg(-2\pi \dim(V_{\chi})\Vol(X)\int_{0}^{s}P_{\sigma}(t)dt\bigg).
\end{equation}
\item  $\bf {case (b)}$ the symmetrized zeta function has the representation
\begin{equation}
S(s;\sigma,\chi)={\det}_{\gr}(A_{\chi}^{\sharp}(\sigma)+s^{2})\exp\bigg(-4\pi \dim(V_{\chi})\Vol(X)\int_{0}^{s}P_{\sigma}(t)dt\bigg).
\end{equation}
\end{enumerate}
\end{thm}
\begin{proof}
By \cite[Lemma 7.1]{Spilioti2018}, \cite[eq. (7.9)]{Spilioti2018} and arguing as in the proof of Proposition 7.7 in
\cite[p. 187-188]{Spilioti2018}, we obtain 
\begin{align*}
  \Tr_{s} \prod_{i=1}^{N}(A_{\chi}^{\sharp}(\sigma)+s_{i}^{2})^{-1}=&\int_{0}^{\infty}\sum_{i=1}^{N}\bigg(\prod_{\substack{j=1\\ j\neq i}}^{N}\frac{1}{s_{j}^{2}-s_{i}^{2}}\bigg)
e^{-ts_{i}^{2}}\Tr_{s}(e^{-t{A_{\chi}^{\sharp}(\sigma)}})dt.\\
\end{align*}
By \cite[Lemma 7.3]{Spilioti2018} and (7.1), the integral 
in the right-hand side of the equation above is absolutely convergent.
We have
\begin{align}
\int_{0}^{\infty}\sum_{i=1}^{N}&\notag\bigg(\prod_{\substack{j=1\\ j\neq i}}^{N}\frac{1}{s_{j}^{2}-s_{i}^{2}}\bigg)e^{-ts_{i}^{2}}\Tr_{s}(e^{-tA_{\chi}^{\sharp}(\sigma)})dt\\\notag
&=\int_{0}^{\infty}\sum_{i=1}^{N}\bigg(\prod_{\substack{j=1\\ j\neq i}}^{N}\frac{1}{s_{j}^{2}-s_{i}^{2}}\bigg)
\frac{1}{2s_{i}t}\bigg(-\frac{d}{ds_{i}}e^{-ts_{i}^{2}}\bigg)\Tr_{s}(e^{-tA_{\chi}^{\sharp}(\sigma)})dt.\\
\end{align}
The right-hand side of (7.15) gives
\begin{align*}
&\int_{0}^{\infty}\sum_{i=1}^{N}\bigg(\prod_{\substack{j=1\\ j\neq i}}^{N}\frac{1}{s_{j}^{2}-s_{i}^{2}}\bigg)
\frac{1}{2s_{i}t}\bigg(-\frac{d}{ds_{i}}e^{-ts_{i}^{2}}\bigg)\T_{s}r(e^{-tA_{\chi}^{\sharp}(\sigma)})dt\\
&=\lim_{z\rightarrow 0}\int_{0}^{\infty}\sum_{i=1}^{N}\bigg(\prod_{\substack{j=1\\ j\neq i}}^{N}\frac{1}{s_{j}^{2}-s_{i}^{2}}\bigg)
\frac{1}{2s_{i}}\bigg(-\frac{d}{ds_{i}}e^{-ts_{i}^{2}}\bigg)t^{z-1}\Tr_{s}(e^{-tA_{\chi}^{\sharp}(\sigma)})dt\\
&=\lim_{z\rightarrow 0}\sum_{i=1}^{N}\bigg(\prod_{\substack{j=1\\ j\neq i}}^{N}\frac{1}{s_{j}^{2}-s_{i}^{2}}\bigg)\frac{1}{2s_{i}}\frac{d}{ds_{i}}\big(-\int_{0}^{\infty}e^{-ts_{i}^{2}}t^{z-1}\Tr_{s}(e^{-tA_{\chi}^{\sharp}(\sigma)})dt\big)\\
&=\sum_{i=1}^{N}\bigg(\prod_{\substack{j=1\\ j\neq i}}^{N}\frac{1}{s_{j}^{2}-s_{i}^{2}}\bigg)\frac{1}{2s_{i}}\frac{d}{ds_{i}}\big(-\xi(0,s_{i}^{2};\sigma)\big)\\
&=\sum_{i=1}^{N}\bigg(\prod_{\substack{j=1\\ j\neq i}}^{N}\frac{1}{s_{j}^{2}-s_{i}^{2}}\bigg)\frac{1}{2s_{i}}\frac{d}{ds_{i}}\big(\log({\det}_{\gr}(A_{\chi}^{\sharp}(\sigma)+s_{i}^{2}))\big),
\end{align*}
where in the last equation we used (7.12).
Therefore, (7.15) becomes
\begin{align}
\int_{0}^{\infty}\sum_{i=1}^{N}&\notag\bigg(\prod_{\substack{j=1\\ j\neq i}}^{N}\frac{1}{s_{j}^{2}-s_{i}^{2}}\bigg)e^{-ts_{i}^{2}}\Tr_{s}(e^{-tA_{\chi}^{\sharp}(\sigma)})dt\\
&= \sum_{i=1}^{N}\bigg(\prod_{\substack{j=1\\ j\neq i}}^{N}\frac{1}{s_{j}^{2}-s_{i}^{2}}\bigg)\frac{1}{2s_{i}}\frac{d}{ds_{i}}\log({\det}_{\gr}(A_{\chi}^{\sharp}(\sigma)+s_{i}^{2})).
\end{align}
We treat here the case (b). One can proceed similarly for the case (a).
The left-hand side of (7.16) can be developed more. Recall that $L(\gamma;\sigma,\chi)$ is given by (3.5).
We insert the right-hand side of the trace formula \cite[eq. (5.39), Theorem 5.5]{Spilioti2018} for the operator $e^{-tA_{\chi}^{\sharp}(\sigma)}$. As before, we consider the super trace of the operator $e^{-tA_{\chi}^{\sharp}(\sigma)}$.
\begin{align*}
\Tr_{s}(e^{-tA_{\chi}^{\sharp}(\sigma)})=&2\dim(V_{\chi})\Vol(X)\int_{\R}e^{-t\lambda^{2}}P_{\sigma}(i\lambda)d\lambda\\
&+\sum_{[\gamma]\neq e}\frac{l(\gamma)}{n_{\Gamma}(\gamma)}L(\gamma;\sigma+w\sigma,\chi)
\frac{e^{-l(\gamma)^{{2}}/4t}}{(4\pi t)^{1/2}}.
\end{align*}
Then, the left-hand side of (7.16) becomes
\begin{align*}
\int_{0}^{\infty}\sum_{i=1}^{N}\bigg(\prod_{\substack{j=1\\ j\neq i}}^{N}\frac{1}{s_{j}^{2}-s_{i}^{2}}\bigg)&e^{-ts_{i}^{2}}\Tr_{s}(e^{-tA_{\chi}^{\sharp}(\sigma)})dt
=\int_{0}^{\infty}\sum_{i=1}^{N}\bigg(\prod_{\substack{j=1\\ j\neq i}}^{N}\frac{1}{s_{j}^{2}-s_{i}^{2}}\bigg)\\
&e^{-ts_{i}^{2}}\bigg(2\dim(V_{\chi})\Vol(X)\int_{\R}e^{-t\lambda^{2}}P_{\sigma}(i\lambda)d\lambda\\
&+\sum_{[\gamma]\neq e} \frac{l(\gamma)}{n_{\Gamma}(\gamma)}L(\gamma;\sigma+w\sigma,\chi)
\frac{e^{-l(\gamma)^{{2}}/4t}}{(4\pi t)^{1/2}}\bigg)dt.
\end{align*}
By \cite[Lemma 7.4]{Spilioti2018}, we can interchange the order of integration for the double integral
\begin{equation*}
\int_{0}^{\infty}\int_{\R}\sum_{i=1}^{N}\bigg(\prod_{\substack{j=1\\ j\neq i}}^{N}\frac{1}{s_{j}^{2}-s_{i}^{2}}\bigg)e^{-ts_{i}^{2}}e^{-t\lambda^{2}}P_{\sigma}(i\lambda)d\lambda.
\end{equation*}
We use the Cauchy integral formula to calculate this integral.
For the calculation of the integral that corresponds to the hyperbolic contribution, we use the identity 
\begin{equation*}
\int_{0}^{\infty}e^{-ts^{2}} \frac{e^{-l(\gamma)^{{2}}/4t}}{(4\pi t)^{1/2}}dt=\frac{1}{2s}e^{-sl(\gamma)}
\end{equation*}
(\cite[eq. (27), p. 146]{Er}).
Hence,
\begin{align*}
\int_{0}^{\infty}\sum_{i=1}^{N}\bigg(\prod_{\substack{j=1\\ j\neq i}}^{N}\frac{1}{s_{j}^{2}-s_{i}^{2}}\bigg)&e^{-ts_{i}^{2}}\bigg(2\dim(V_{\chi})\Vol(X)\int_{\R}e^{-t\lambda^{2}}P_{\sigma}(i\lambda)d\lambda\\
&+\sum_{[\gamma]\neq e} \frac{l(\gamma)}{n_{\Gamma}(\gamma)}L(\gamma;\sigma+w\sigma,\chi)
\frac{e^{-l(\gamma)^{{2}}/4t}}{(4\pi t)^{1/2}}\bigg)dt\\
&=\sum_{i=1}^{N}\bigg(\prod_{\substack{j=1\\ j\neq i}}^{N}\frac{1}{s_{j}^{2}-s_{i}^{2}}\bigg)\frac{2\pi}{s_{i}} \dim(V_{\chi})\Vol(X)P_{\sigma}(s_{i})\\
&+\sum_{i=1}^{N}\bigg(\prod_{\substack{j=1\\ j\neq i}}^{N}\frac{1}{s_{j}^{2}-s_{i}^{2}}\bigg)\frac{1}{2s_{i}}\sum_{[\gamma]\neq e} \frac{l(\gamma)}
{n_{\Gamma}(\gamma)}
L(\gamma;\sigma+w\sigma,\chi)e^{-s_{i}l(\gamma)}.
\end{align*}
Hence, by (7.16), we get
\begin{align}
\sum_{i=1}^{N}\bigg(\prod_{\substack{j=1\\ j\neq i}}^{N}\frac{1}{s_{j}^{2}-s_{i}^{2}}\bigg)&\frac{1}{2s_{i}}\frac{d}{ds_{i}}\notag\log({\det}_{\gr}(A_{\chi}^{\sharp}(\sigma)+s_{i}^{2}))\\\notag
&=\sum_{i=1}^{N}\bigg(\prod_{\substack{j=1\\ j\neq i}}^{N}\frac{1}{s_{j}^{2}-s_{i}^{2}}\bigg)\frac{2\pi}{s_{i}}\dim(V_{\chi})\Vol(X)P_{\sigma}(s_{i})\\
&+\sum_{i=1}^{N}\bigg(\prod_{\substack{j=1\\ j\neq i}}^{N}\frac{1}{s_{j}^{2}-s_{i}^{2}}\bigg)\frac{1}{2s_{i}}
\sum_{[\gamma]\neq e} \frac{l(\gamma)}{n_{\Gamma}(\gamma)}L(\gamma;\sigma+w\sigma,\chi)
e^{-s_{i}l(\gamma)}.
\end{align}
We fix now the variables $s_{2},\ldots, s_{N}\in\C$ and let the variable $s_{1}=s\in\C$ vary.
We multiply both sides of (7.17) by 
\begin{equation*}
2s\prod_{j=2}^{N}s_{j}^{2}-s^{2}.
\end{equation*}
Then, we get
\begin{align}
\frac{d}{ds}\log({\det}_{\gr}(A_{\chi}^{\sharp}(\sigma)+s^{2}))=\notag&4\pi\dim(V_{\chi})\Vol(X)P_{\sigma}(s)\\\notag
&+\sum_{[\gamma]\neq e} \frac{l(\gamma)}{n_{\Gamma}(\gamma)}L(\gamma;\sigma+w\sigma,\chi)e^{-sl(\gamma)}\\
&+K'(s),
\end{align}
where $K'(s)$ is a certain odd  polynomial, which is of the from
\begin{equation*}
K'(s)=2s\prod_{j=2}^{N}(s_{j}^{2}-s^{2})Q(s_{2},\ldots,s_{N}).
\end{equation*}
The term $Q(s_{2},\ldots,s_{N})$ comes from the terms that correspond to the summands over $i=2,\ldots,N$ and hence it has a fixed value in $\C$, since $s_{2},\ldots,s_{N}$ are fixed.

Next, we can substitute the term in (7.18), that comes from the hyperbolic contribution of the trace formula, with the logarithmic derivative of the symmetrized zeta function. 
By (3.6), we have
\begin{align*}
\frac{d}{ds}\log({\det}_{\gr}(A_{\chi}^{\sharp}(\sigma)+s^{2}))=&4\pi\dim(V_{\chi})\Vol(X)P_{\sigma}(s)\\\notag
&+L_{S}(s)+K'(s).
\end{align*}
We integrate with respect to $s$ and get
\begin{align*}
\log({\det}_{\gr}(A_{\chi}^{\sharp}(\sigma)+s^{2}))=&4\pi\dim(V_{\chi})\Vol(X)\int_{0}^{s}P_{\sigma}(t)dt\\\notag
&+ \log S(s;\sigma,\chi)+K(s).
\end{align*}
Hence,
\begin{align}
\log S(s;\sigma,\chi)=\notag&\log({\det}_{\gr}(A_{\chi}^{\sharp}(\sigma)+s^{2}))-K(s)\\
& -4\pi\dim(V_{\chi})\Vol(X)\int_{0}^{s}P_{\sigma}(t)dt.
\end{align}
We want to show that $K(s)=0$.
To this end, we study the asymptotic behaviour of all terms in (7.19), as $s\rightarrow \infty$.
By Lemma 4.2, there exist coefficients $c_{j}$ such that 
\begin{equation*}
 \Tr_{s}(e^{-tA^{\sharp}_{\chi}(\sigma)})\sim_{t\rightarrow 0}\sum_{j=0}^{\infty}c_{j}t^{\frac{j-d}{2}}.
\end{equation*}
We use the asymptotic formula (13) from \cite{Quine}. 
By \cite[Lemma 1.7.4 (d)]{Gilk}, $c_{j}$ vanish for $j$ odd.
Hence, in our case, since $d$ is odd, the first sum over the integers in the 
right-hand side of the asymptotic formula 
(13) in \cite{Quine} does not contribute.
We have
\begin{align}
-\log{\det}_{\gr}(A_{\chi}^{\sharp}(\sigma)+s^{2})\sim_{s\rightarrow\infty}\sum_{k=0}^{\infty}c_{2k}\Gamma(2k)s^{d-2k}.
\end{align}
The estimation of the sum on the right hand side of (3.14) on p. 163 of \cite{Spilioti2018} can be improved to show that the $\log$ of the Selberg zeta function is exponentially decreasing. Therefore the $\log$ of the symmetrized zeta function also decreases exponentially.
Hence, as $s\rightarrow\infty$, the left-hand side of 
(7.19) goes to zero. On the other hand,
as $s\rightarrow\infty$, the right-hand side of (7.20) contains only odd powers of $s$,
the term that involves the Plancherel polynomial in the left hand side of (7.19)
is an odd polynomial, and $K(s)$ is an even polynomial. 
Hence, $K(s)$ vanish identically.
Moreover, we conclude
\begin{equation*}
 -4\pi\dim(V_{\chi})\Vol(X)\int_{0}^{s}P_{\sigma}(t)dt=\sum_{k=0}^{(d-1)/2}c_{2k}\Gamma(2k)s^{d-2k}
\end{equation*}
and $c_{2k}=0$ for $2k>d$.
\end{proof}
We prove now a determinant formula for the Ruelle zeta function.
Recall form Section 6 the standard representation $\sigma_{p}$ of $M$ in $\Lambda^{p}\R^{d-1}\otimes\C$.
Let $\sigma,\sigma'\in\widehat{M}$. We denote by $[\sigma_{p}\otimes\sigma:\sigma']$
the multiplicity of $\sigma'$ in $\sigma_{p}\otimes\sigma$.
We distinguish again two cases for $\sigma'\in\widehat{M}$. 
\begin{itemize}
 \item  {\bf case (a)}: \textit{$\sigma'$ is invariant under the action of the restricted Weyl group $W_{A}$.}
Then, $i^{*}(\tau)=\sigma'$, where $\tau\in R(K)$.
 \item {\bf case (b)}: \textit{$\sigma'$ is not invariant under the action of the restricted Weyl group $W_{A}$.}
Then, $i^{*}(\tau)=\sigma'+w\sigma'$, where $\tau\in R(K)$.
\end{itemize}
We define the operator
We define the operator
\begin{equation*}
 A_{\chi}^{\sharp}(\sigma_{p}\otimes\sigma):=\bigoplus_{[\sigma']\in\widehat{M}/W_{A}}\bigoplus_{i=1}^{[\sigma_{p}\otimes\sigma:\sigma']}A^{\sharp}_{\chi}(\sigma'),
\end{equation*}
acting on the space $C^{\infty}(X,E(\sigma')\otimes E_{\chi})$, where $E(\sigma')$ is the vector bundle 
over $X$, constructed as in \cite[p. 175]{Spilioti2018}.

Let $\alpha$ be the unique positive root of $(\mathfrak{g},\mathfrak{a})$. Let $H\in\mathfrak{a}$ such that 
$\alpha(H)=1$. Recall that the character $\lambda\equiv\lambda_{p}$ of $A$ is defined by $\lambda\equiv\lambda_{p}(a)=e^{p\alpha(\log a)}$. Then,
we can identify $\lambda$ with $p$.
\begin{prop}
The Ruelle zeta function has the representation 
\begin{itemize}
 \item  {\bf case (a)}
\begin{align}
R(s;\sigma,\chi)=\notag&\prod_{p=0}^{d-1}{{\det}_{\gr}(A^{\sharp}_{\chi}(\sigma_{p}\otimes\sigma)+(s+\vert\rho\vert-p)^{2})}^{(-1)^{p}}\\
&\exp\bigg((-1)^{\frac{d-1}{2}+1}\pi(d+1)\dim(V_{\sigma})\dim (V_{\chi})\frac{\Vol(X)}{\Vol(S^{d})}s\bigg).
\end{align}
 \item {\bf case (b)}
\begin{align}
 R(s;\sigma,\chi)R(s;w\sigma,\chi)=\notag&\prod_{p=0}^{d-1}{{\det}_{\gr}(A^{\sharp}_{\chi}(\sigma_{p}\otimes\sigma)+(s+\vert\rho\vert-p)^{2})}^{(-1)^{p}}\\
&\exp\bigg((-1)^{\frac{d-1}{2}+1}2\pi(d+1)\dim(V_{\sigma})\dim (V_{\chi})\frac{\Vol(X)}{\Vol(S^{d})}s\bigg).
\end{align}
\end{itemize}
\end{prop}
\begin{proof}
We prove the assertion for case (b). One can proceed similarly for case (a).
By \cite[Theorem 7.12]{Spilioti2018}, we have 
the expression of the Ruelle zeta function as a product of 
Selberg zeta functions. Then, we see
\begin{align*}
 R(s;\sigma,\chi) R(s;w\sigma,\chi)&=\prod_{p=0}^{d-1}Z(s+\vert\rho\vert-p;\sigma_{p}\otimes\sigma,\chi)^{(-1)^{p}}\prod_{p=0}^{d-1}Z(s+\vert\rho\vert-p;\sigma_{p}\otimes w\sigma,\chi)^{(-1)^{p}}\\
&=\prod_{p=0}^{d-1}S(s+\vert\rho\vert-p;\sigma_{p}\otimes\sigma,\chi)^{(-1)^{p}}.
\end{align*}
Hence, if we equip the determinant formula for the symmetrized zeta function (Theorem 7.8.\textit{2}), we have
\begin{align}
  R(s;\sigma,\chi) R(s;w\sigma,\chi)\notag&=\prod_{p=0}^{d-1}{{\det}_{\gr}(A^{\sharp}_{\chi}(\sigma_{p}\otimes\sigma)+(s+\vert\rho\vert-p)^{2})}^{(-1)^{p}}\\
&\exp\bigg(\sum_{p=0}^{d-1}(-1)^{p}(-4\pi \dim(V_{\chi})
\Vol(X))\int_{0}^{s+\vert\rho\vert-p}P_{\sigma_{p}\otimes\sigma}(t)dt\bigg).
\end{align}
On the other hand,
\begin{equation*}
\sum_{p=0}^{d-1}(-1)^{p} \int_{0}^{s+\vert\rho\vert-p}Q_{\sigma_{p}\otimes\sigma}(t)dt=\int_{0}^{s}f(t)dt,
\end{equation*}
where $f(t)$ is defined by (6.2).
Therefore, as in the proof of Theorem 6.2, 
\begin{equation}
 \sum_{p=0}^{d-1}(-1)^{p}\int_{0}^{s+\vert\rho\vert-p}Q_{\sigma_{p}\otimes\sigma}(t)dt=(d+1)\dim(V_{\sigma})s.
\end{equation}
Hence, (7.23) becomes by (7.24)
\begin{align*}
 R(s;\sigma,\chi) R(s;w\sigma,\chi)&=\prod_{p=0}^{d-1}{{\det}_{\gr}(A^{\sharp}_{\chi}(\sigma_{p}\otimes\sigma)+(s+\vert\rho\vert-p)^{2})}^{(-1)^{p}}\\
&\exp\bigg(4\pi c(n)(d+1)\dim(V_{\chi})\dim(V_{\sigma})\Vol(X)s\bigg),
\end{align*}
where $c(n)$ is as in (2.9).
The assertion follows.
\end{proof}

\paragraph{Acknowledgements}
The author is grateful to her supervisor Werner M\"{u}ller for his guidance and suggestions,
as this article is based on the second part of her PhD thesis.

\bibliographystyle{amsalpha}
\bibliography{ref}
\contact

\end{document}